\documentclass[12pt]{amsart}

\usepackage{amssymb,comment}
\usepackage{hyperref}

\usepackage{enumitem}
\setlist[enumerate,1]{label={(\alph*)}}
\setlist[enumerate,2]{label={(\roman*)}}
\usepackage{bigints}
\usepackage{graphicx}

\newtheorem{tm}{Theorem}[section]
\newtheorem{pr}[tm]{Proposition}
\newtheorem{lm}[tm]{Lemma}
\newtheorem{cy}[tm]{Corollary}

\theoremstyle{definition}
\newtheorem{df}[tm]{Definition}

\theoremstyle{remark}
\newtheorem{rem}[tm]{Remark}

\newtheorem{qn}[tm]{Question}

\newcommand{\Z}{{\mathbb Z}}
\newcommand{\C}{{\mathbb C}}
\newcommand{\N}{{\mathbb N}}

\newcommand{\R}{{\mathbb R}}

\renewcommand{\tilde}{\widetilde}
\renewcommand{\bar}{\overline}
\renewcommand{\phi}{\varphi}
\renewcommand{\epsilon}{\varepsilon}
\renewcommand{\geq}{\geqslant}
\renewcommand{\leq}{\leqslant}
\renewcommand{\max}{{\rm max}}

\newcommand{\Id}{\operatorname{Id}}
\newcommand{\id}{\operatorname{\text{\sf id}}}

\renewcommand{\Re}{\operatorname{Re}}
\renewcommand{\Im}{\operatorname{Im}}

\newcommand{\M}{{\mathcal M}}
\newcommand{\bM}{{\overline{\mathcal M}}}
\newcommand{\z}{{\vec z}}
\newcommand{\zz}{{\mathbf z}}
\DeclareMathOperator{\dvol}{{dvol}}
\DeclareMathOperator{\PSL}{{PSL}}

\newcommand{\LL}{\mathcal L}
\newcommand{\otimesc}{\otimes}
\DeclareMathOperator{\ob}{Ob}
\newcommand{\A}{{\mathcal A}}
\newcommand{\B}{{\mathcal B}}

\newcommand{\J}{{\mathcal J}}
\DeclareMathOperator{\dist}{{dist}}
\renewcommand{\i}{\sqrt{-1}}
\DeclareMathOperator{\graph}{{graph}}

\newcommand{\e}{\mathbf{e}}
\DeclareMathOperator{\Map}{Map}

\newcommand{\I}{\mathcal{I}}
\newcommand{\D}{\mathcal{D}}

\begin{document}
\title{Locality in the Fukaya category of a hyperk\"ahler manifold}
\author{Jake P. Solomon, Misha Verbitsky}

\date{July 2019}

\begin{abstract}
Let $(M,I,J,K,g)$ be a hyperk\"ahler manifold. Then the complex manifold $(M,I)$ is holomorphic symplectic. We prove that for all real $x, y,$  with $x^2 + y^2 = 1$ except
countably many, any finite energy $(xJ+yK)$-holomorphic curve with boundary in a collection of $I$-holomorphic Lagrangians must be constant. By an argument based on the \L ojasiewicz inequality, this result holds no matter how the Lagrangians intersect each other. It follows that one can choose perturbations such that the holomorphic polygons of the associated Fukaya category lie in an arbitrarily small neighborhood of the Lagrangians. That is, the Fukaya category is local. We show that holomorphic Lagrangians are tautologically unobstructed in the sense of Fukaya-Oh-Ohta-Ono. Moreover, the Fukaya $A_\infty$ algebra of a holomorphic Lagrangian is formal. Our result also explains why the special Lagrangian condition holds without instanton corrections for holomorphic Lagrangians.
\end{abstract}

\maketitle

\pagestyle{plain}

\tableofcontents


\section{Introduction}\label{sec:intro}
\subsection{Background}

Before stating our results, we recall some basic facts about hyperk\"ahler manifolds.
For more details and references, see \cite{_Besse:Einst_Manifo_}.

\begin{df} A {\bf hypercomplex manifold}
is a manifold $M$ with three complex structures
$I, J, K$, satisfying quaternionic relations
\[
IJ = - JI = K, \ \ I^2 = J^2=K^2=-\Id_{TM}.
\]
A {\bf hyperk\"ahler manifold}
is a hypercomplex manifold equipped with a metric $g$
which is K\"ahler with respect to $I,J,K$.
\end{df}

\begin{df} A {\bf holomorphic symplectic manifold}
is a manifold $M$ with a complex structure $I$ and a closed non-degenerate holomorphic $(2,0)$-form $\Omega$. A complex submanifold $L \subset M$ is {\bf holomorphic Lagrangian} if $\Omega|_L = 0.$
\end{df}

\begin{rem}\label{rem:hkhs} A hyperk\"ahler manifold $M$ is
equipped with three symplectic forms $\omega_I$, $\omega_J$,
$\omega_K$. The form
$\Omega_I:= \omega_J+\sqrt{-1}\omega_K$
is a holomorphic symplectic 2-form on
$(M,I)$.
\end{rem}

The following result is a consequence of the Calabi-Yau theorem~\cite{_Ya_}
and the Bochner technique.

\begin{tm}
A compact, K\"ahler, holomorphically symplectic manifold
admits a unique hyperk\"ahler metric in any K\"ahler class.
\end{tm}

In other words, ``hyperk\"ahler'' in the context of compact manifolds
is essentially synonymous with ``holomorphic symplectic.'' Examples of non-compact hyperk\"ahler manifolds include the moduli space of Higgs bundles~\cite{_Hit_} and Nakajima quiver varieties~\cite{_Na_}. The Fukaya categories of such manifolds have drawn considerable attention~\cite{_KW_,_SS_}.

Let $\Theta$ be an almost complex structure on a manifold $M,$ that is, an automorphism of $TM$ satisfying $\Theta^2 = -\Id.$ Let $\Sigma$ be a Riemann surface with complex structure $j$, possibly with boundary, and not necessarily compact.

\begin{df}
A map $u: \Sigma \to M$ is called $\Theta$-{\bf holomorphic} if
\[
\bar\partial_\Theta u : = \Theta\circ du + du \circ j = 0.
\]
The {\bf energy} of $u$ with respect to a Riemannian metric $g$ on $M$ is given by
\[
E_g(u) = \int_\Sigma |du|^2_g\dvol.
\]
\end{df}

\subsection{Main Results}
We call a Riemann surface $\Sigma$ \textbf{nice} if it can be obtained by removing a finite number of boundary points from a compact connected Riemann surface with boundary $\overline\Sigma.$ We do not exclude the case $\partial\Sigma = \emptyset.$ Let
\[
\partial \Sigma_i,\qquad i = 1,\ldots,b(\Sigma),
\]
denote the connected components of $\partial\Sigma.$ Let
\[
\chi_c(\Sigma) = \chi\!\left(\overline \Sigma\right) - \left|\overline \Sigma \setminus \Sigma\right|
\]
denote the compactly supported Euler characteristic of $\Sigma.$
Let $(M,\omega)$ be a symplectic manifold, and let $\LL$ be a collection of Lagrangian submanifolds
$L_i\subset M,\; i \in A,$ where $A$ is a countable index set. An \textbf{$A$-labeling} of a nice Riemann surface $\Sigma$ is a function $l : \{1,\ldots,b(\Sigma)\} \to A.$
An almost complex structure $\Theta$ on $M$ is called {\bf $\omega$-tame} if for all tangent vectors $0 \neq \xi \in T_pM,$ we have $\omega(\xi,\Theta\xi) > 0.$ Let $\J_{\omega}$ denote the collection of all $\omega$-tame almost complex structures.
Assume there exists a compact set $N \subset M$ such that
\[
\dist_g(L_i\setminus N,\,L_j \setminus N) > 0
\]
for all $i,j.$ In particular, it suffices for $M$ to be compact.
Given a hypercomplex structure $(I,J,K)$ on $M,$ let $R_I\simeq S^1$ denote the space of complex structures on $M$
of the form $x J + y K$ for $x^2 + y^2=1$. The main result of this paper is the following theorem.

\begin{tm}\label{tm:main}
Suppose there exists a complete hyperk\"ahler structure $(I,J,K,g)$ on $M$ such that $\omega = \omega_J$ and the Lagrangian submanifolds $L_i$ for $i \in A$ are $I$-holomorphic. There exists a countable set $P \subset R_I$ with the following significance. Let $\Sigma$ be a nice Riemann surface and let $l$ be an $A$-labeling. If $\Theta \in R_I \setminus P,$ then a $\Theta$-holomorphic map $u: \Sigma \to M$ with $u(\partial\Sigma_i) \subset L_{l(i)}$ and $E_g(u) < \infty$
is necessarily constant.
\end{tm}

It is crucially important in Theorem~\ref{tm:main} that there is no restriction on how the Lagrangians $L_i$ intersect. Indeed, in the compact setting, moving the $L_i$ by any non-trivial Hamiltonian flow will destroy their holomorphicity. However, for the purposes of Lagrangian Floer cohomology and the Fukaya category, only Hamiltonian flows can be used to move the Lagrangians so they intersect transversely. See the end of Section~\ref{ssec:gf} and the end of Section~\ref{ssec:af}.

On the other hand, the standard construction of Lagrangian Floer cohomology and the Fukaya category uses Hamiltonian flows to move all Lagrangians so they intersect transversely, or at least cleanly. As $I$-holomorphic Lagrangians $L_i$ move under a Hamiltonian flow and loose their holomorphicity, $\Theta$-holomorphic maps with boundary in the $L_i$ can appear for all $\Theta \in R_I.$ Nonetheless, as a consequence of Theorem~\ref{tm:main} and Gromov compactness, we show that the $\Theta$-holomorphic maps that appear can be forced to remain in an arbitrarily small neighborhood of the $L_i$ for all but countably many $\Theta \in R_I.$ This is the content of the following Theorem~\ref{tm:loc}. In Section~\ref{ssec:loc}, Theorem~\ref{tm:loc} is used to show that in a certain sense the Fukaya category generated by $I$-holomorphic Lagrangians in a hyperk\"ahler manifold is local. In Remark~\ref{rem:st}, we discuss how this locality result provides evidence for a connection between the Fukaya category and the sheaf theoretic alternatives proposed by Kapustin~\cite{_K1_,_K2_}, Behrend-Fantechi~\cite{_BF_} and Brav-Bussi-Dupont-Joyce-Szendroi \cite{_BBDJS_}.
\begin{tm}\label{tm:loc}
Suppose $(M,\omega),\,L_i$ and $P,$ are as in Theorem~\ref{tm:main}. Then, for all $\Theta \in (R_I \setminus P)\cap \J_\omega$, we have the following. Choose Hamiltonian flows $\phi_{i,t} : M \to M,$ constants $E_0 \geq 0, \chi_0 \in \Z,$ a finite subset $A_0 \subset A,$ and an open $V \subset M$ such that
\begin{equation}\label{eq:cupLi}
\bigcup_{i \in A_0} L_i \subset V.
\end{equation}
There exists $\epsilon > 0$ such that if $t < \epsilon,$ then for all nice $\Sigma,$ all $A_0$-labelings $l$ and all $\Theta$-holomorphic maps
$
u : \Sigma \to M
$
with
\[
u(\partial \Sigma_i) \subset \phi_{l(i),t}(L_{l(i)}), \qquad E_g(u) \leq E_0, \qquad \chi_c(\Sigma) \geq \chi_0,
\]
we have
\[
u(\Sigma) \subset V.
\]
\end{tm}
\begin{rem}
If we choose the Hamiltonian flows $\phi_{i,t}$ to be real analytic and fix the topological type of $\Sigma$ as well as $l,$ it should be possible to replace hypothesis~\eqref{eq:cupLi} with the weaker hypothesis
\[
\bigcap_{i \in [b(\Sigma)]} L_{l(i)} \subset V.
\]
See Remark~\ref{rem:fgc} for a more in depth discussion.
\end{rem}

Our proof of Theorem~\ref{tm:main} relies on the following result, which is valid in a more general setting.  For $\Theta\in \J_\omega,$ let $g_\Theta$ denote the Riemannian metric defined by
\[
g_\Theta(\eta,\xi) = (\omega(\eta,\Theta\xi) + \omega(\xi,\Theta\eta))/2.
\]
For $h$ a Riemannian metric, let $d_h$ denote the associated distance function.
\begin{tm}\label{tm:conti}
Suppose $M,\omega,$ and the Lagrangian submanifolds $L_i$ for $i  \in A$ are real analytic. Let $\Theta$ be an $\omega$-tame almost complex structure such that the metric $g_\Theta$ is complete. Choose a nice Riemann surface $\Sigma,$ a Hermitian metric $h$ on $\overline \Sigma$ and an $A$-labeling $l.$ Let $u: \Sigma \to M$ be a $\Theta$-holomorphic map with $u(\partial\Sigma_i) \subset L_{l(i)}$ and $E_{g_\Theta}(u) < \infty.$ Then $u$ extends to a continuous map $\bar u : \overline \Sigma \to M.$ Moreover, there exist $c,\alpha,\epsilon >0$ such that for each $p \in \overline \Sigma \setminus \Sigma,$ we have
\begin{equation*}
|du|_{g_\Theta,h}(z) < \frac{c}{d_h(p,z) (\log d_h(p,z))^\alpha}, \qquad z \in \Sigma, \quad d_h(z,p) < \epsilon.
\end{equation*}
\end{tm}
If the Lagrangian submanifolds $L_i$ intersect each other cleanly, then Theorem~\ref{tm:conti} is well known and holds without any real analyticity hypothesis~\cite[Lemma 2.5]{_Wo_}. In the general case, the proof uses \L ojasiewicz's inequalities~\cite{_Loj65_} and through them the real analyticity hypothesis in an essential way. Of course, the holomorphic Lagrangians considered in Theorems~\ref{tm:main} and~\ref{tm:loc} are a fortiori real analytic. The use of \L ojasiewicz's gradient inequality to obtain results in geometric analysis dates back to the work of Simon on harmonic map heat flow and minimal surfaces~\cite{_Si83_}.  Related works include~\cite{_CM_,_MMR_,_Ta_}.

\subsection{Outline of paper}
In Section~\ref{sec:fcfc}, after giving necessary background, we explain the implications of Theorems~\ref{tm:main} and~\ref{tm:loc} for the Fukaya category. In Section~\ref{sec:slag}, we outline how Theorem~\ref{tm:main} explains the unusual abundance of special Lagrangian submanifolds in hyperk\"ahler manifolds.

The proof of Theorem~\ref{tm:conti} is contained in Sections~\ref{sec:lsa}-\ref{sec:limit}. Section~\ref{sec:lsa} uses the Whitney regularity of analytic sets~\cite{_Bi_,_Ha_} to prove that the symplectic action of a sufficiently short path between two analytic Lagrangians is well-defined. Section~\ref{sec:lii} proves an isoperimetric inequality for the symplectic action based on \L ojasiewicz's gradient inequality~\cite{_Loj65_}. Finally, Section~\ref{sec:limit} combines the results of the preceding two sections with properties of holomorphic curves and a control-function argument of \L ojasiewicz~\cite{_Loj84_} to deduce Theorem~\ref{tm:conti}.

The proof of Theorem~\ref{tm:main} is given in Sections~\ref{sec:hc}-\ref{_proof_Section_}. The main idea is that each homotopy class $h$ of maps $u$ as in Theorem~\ref{tm:main} gives rise to a functional $\phi_{h,\Sigma}$ on $R_I.$ The functional $\phi_{h,\Sigma}$ either vanishes identically or attains its maximum at unique $\Theta \in R_I.$ Moreover, if $h$ is represented by a non-constant $\Theta$-holomorphic map, then $\phi_{h,\Sigma}$ must attain its unique maximum at $\Theta.$ So, if we remove the set of $\Theta$ at which $\phi_{h,\Sigma}$ attains its maximum for some $h,\Sigma,$ there are no non-constant $\Theta$-holomorphic maps for all the remaining $\Theta.$ Section~\ref{sec:hc} uses \L ojasiewicz's triangulation theorem for real analytic sets~\cite{_Loj64_} to show the set of homotopy classes $h$ is countable. Furthermore, it establishes a framework within which the functional $\phi_{h,\Sigma}$ can be defined. More specifically, the functional $\phi_{h,\Sigma}$ is given by integrating a K\"ahler form pulled-back by a representative of $h.$ However, since the Lagrangians $L_i$ may intersect each other in quite bad sets, there need not exist smooth homotopies between different representatives. Thus, the integral defining $\phi_{h,\Sigma}$ could a priori depend on the choice of representative of $h$. The proof that $\phi_{h,\Sigma}$ in fact depends only on $h$ is given in Section~\ref{sec:hc} using the properties of the symplectic action proved in Sections~\ref{sec:lsa}-\ref{sec:lii}. Section~\ref{_proof_Section_} combines the results of Section~\ref{sec:hc} and Theorem~\ref{tm:conti} to complete the proof of Theorem~\ref{tm:main}. It concludes with the proof of Theorem~\ref{tm:loc}.

A result analogous to Theorem~\ref{tm:main} for maps from closed Riemann surfaces to compact hyperk\"ahler manifolds was proved in~\cite{_Verbitsky:trianal_} using Hodge theory. This result was generalized to the non-compact case in~\cite{_Verbitsky:non-compact_subva_}. The present work builds on the approach of~\cite{_Verbitsky:non-compact_subva_}.

\section{Floer cohomology and the Fukaya category}\label{sec:fcfc}
In this section, we describe how Theorem~\ref{tm:main} implies Floer cohomology and the Fukaya category are particularly well behaved for holomorphic Lagrangian submanifolds of a hyperk\"ahler manifold.

\subsection{Geometric framework}\label{ssec:gf}

We start with a brief overview of Floer cohomology and Fukaya $A_\infty$ operations for Lagrangian submanifolds with vanishing Maslov class. We refer the reader to~\cite{_FOOO_,_Se_} for a detailed treatment.
Let
\[
\Lambda = \left.\left\{\sum  a_i T^{E_i}\right|a_i \in \R,\; E_i \in \R, \; \lim_{i \to \infty} E_i = \infty\right\}
\]
denote the Novikov field.
Let $\| \cdot \| : \Lambda \to \R_{>0}$ be the non-Archimedean norm given by $\| 0 \| = 0$ and
\begin{equation}\label{eq:norm}
\left\|\sum_i a_i T^{E_i}\right\| = \exp(-\min_i E_i).
\end{equation}
Let $(M,\omega)$ be a symplectic manifold of dimension $2n,$ either compact or with appropriately bounded geometry. Let $\Theta \in \J_\omega.$
Assume there exists a non-vanishing complex $n$-form $\Phi$ that is of type $(n,0)$ with respect to $\Theta.$ A {\bf grading} for an oriented Lagrangian submanifold $L \subset M$ is the choice of a function $\theta_L : L \to \R$ such that $\Phi|_L = e^{\sqrt{-1}\pi\theta_L}\rho$ where $\rho$ is a positive real $n$-form on $L.$
Consider a collection $\LL$ of graded spin compact Lagrangian submanifolds $L_i \subset M$ for $i$ belonging to an index set $A.$ Assume that the $L_i$ intersect each other cleanly.  That is, for $B \subset A,$ the intersection $L_B = \cap_{i\in B}L_i$ is a smooth manifold, and
\[
TL_B = \bigcap_{i \in B} TL_i.
\]

For each connected component $C \subset L_i\cap L_j,$ the gradings on $L_i, L_j,$ give rise to an integer $m_C,$ called the Maslov index, in the following way. Let $p \in C.$ Choose a direct sum decomposition $T_pL_i = \oplus_k V_k$ and $\alpha_k \in [0,1)$ such that $T_pL_j = \oplus_k e^{\pi \alpha_k \Theta} V_k.$
Then,
\begin{equation*}
m_C = \sum_{k} \alpha_k \dim V_k - \theta_{L_j}(p) + \theta_{L_i}(p).
\end{equation*}
One verifies that $m_C$ does not depend on $p$ or the direct sum decomposition.

We define the Floer complex $CF^*(L_i,L_j)$ as follows. For $N$ a smooth manifold, let $A^k(N)$ denote the differential $k$-forms on~$N.$ Define
\begin{equation*}
CF^*(L_i,L_j) = \bigoplus_{C \subset L_i\cap L_j}  A^*(C)[-m_C] \otimes \Lambda.
\end{equation*}
The norm $\|\cdot \|$ on $\Lambda$ induces a norm, also denoted $\|\cdot \|$, on $CF^*(L_i,L_j),$ which makes $CF^*(L_i,L_j)$ a normed $\Lambda$ vector space. Here and below, all normed vector spaces are understood to be completed with respect to the norm.  If $\alpha \in CF^m(L_i,L_j)$ for some $m,$ write $|\alpha| = m.$

The vector spaces $CF^*(L_i,L_j)$ come with a family of multilinear operations constructed using $\Theta$-holomorphic curves. Let $D \subset \C$ denote the closed unit disk equipped with the complex orientation, and let $\Delta \subset (\partial D)^{k+1}$ denote the pairwise diagonal. For
\[
\z = (z_0,\ldots,z_k) \in (\partial D)^{k+1} \setminus \Delta,
\]
let $\zz$ denote the corresponding subset of $\partial D.$ Say that $\z$ is cyclically ordered if its order agrees with the cyclic order given by the induced orientation of $\partial D.$ Let $(z_i,z_{i+1}) \subset \partial D$ denote the open interval starting at $z_i$ and ending at $z_{i+1}$ with respect to the induced orientation of $\partial D.$ For $\I = (i_0,\ldots,i_k) \in A^{k+1}$ and $E \in \R,$ define
\[
\widetilde\M(\I,E,\LL) = \left\{(u,\z)\left|
\begin{array}{l}
\z \in (\partial D)^{k+1} \setminus\Delta, \text{ cyclically ordered,}\\
u : D\setminus \zz \to M  \text{ smooth},\\
\bar \partial_\Theta u = 0,\\
u((z_j,z_{j+1})) \subset L_{i_j},\\
\int u^*\omega = E.
\end{array}
\right.\right\},
\]
For $(u,\z) \in \widetilde\M(\I,E,\LL),$ we have~\cite[Lem. 2.2.1]{_MS_}
\[
E = \int u^* \omega = \frac{1}{2}\int |du|_{g_\Theta}^2 \dvol_D \geq 0.
\]
So, if $E < 0$, then $\widetilde \M(\I,E,\LL) = \emptyset.$ Furthermore, if $E = 0$ then $\widetilde \M(\I,E,\LL)$ consists of constant maps.

An automorphism of the disk $\phi \in \PSL(2,\R)$ acts on
\[
(u,\z) \in \widetilde \M(\I,E,\LL)
\]
by $u \mapsto u \circ \phi$ and $z_j \mapsto \phi^{-1}(z_j).$ A pair $(I,E)$ is called $\emph{stable}$ if either $E > 0$ or $k \geq 2$ and $E =0.$ If $(I,E)$ is stable, then the action of $\PSL(2,\R)$ on $\M(\I,E,\LL)$ has finite order stabilizers. For stable $(I,E)$, set $\M(\I,E,\LL) = \widetilde \M(\I,E,\LL)/\PSL(2,\R)$. Otherwise, set $\M(\I,E,\LL) = \emptyset.$
For $[u,\z] \in \M(\I,E,\LL)$, the condition
\[
\int u^*\omega = E < \infty
\]
and the assumption that the Lagrangians $L_i$ intersect cleanly, imply that $u$ extends to a continuous map $\bar u : D \to M.$ See~\cite[Lemma 2.5]{_Wo_}. So, we define evaluation maps
\[
ev_j^E : \M(\I,E,\LL) \to L_{i_{j-1}} \cap L_{i_j}, \quad j = 0,\ldots,k,
\]
by $ev_j^E((u,\z)) = \bar u(z_j).$ The spaces $\M(\I,E,\LL)$ admit stable map compactifications $\bM(\I,E,\LL)$ to which the evaluation maps extend naturally, and we denote these extensions by $ev_j^E$ as well.

Although moduli spaces of stable maps are usually singular, it can been shown in many cases that they admit a virtual fundamental chain, which behaves like the fundamental chain of a smooth manifold. For the purpose of this discussion, we assume that operations of virtual pull-back (resp.~push-forward) of differential forms along the maps $ev_i^E$ (resp.~the map $ev_0^E$), have been defined. These operations should satisfy the same properties as push-forward and pull-back along smooth maps of manifolds, with $ev_0^E$ behaving like a proper submersion. When the moduli spaces are smooth of expected dimension, virtual pull-back and push-forward reduce to the usual pull-back and push-forward. Such operations were introduced in a similar context by Fukaya~\cite{_Fu_}. It is expected that polyfolds~\cite{_poly_} will give an alternate approach. We implicitly extend all operations on differential forms $\Lambda$-linearly to the Floer complex $CF^*(L_i,L_j)$.

Define maps
\[
\mu_k :CF^*(L_{i_0},L_{i_1})\otimesc \cdots \otimesc CF^*(L_{i_{k-1}},L_{i_k}) \to CF^*(L_{i_0},L_{i_k})[2-k]
\]
by
\begin{multline*}
\mu_k (\alpha_1,\ldots,\alpha_k) = \\
 = \delta_{k,1}d\alpha_1 + (-1)^{\epsilon(\alpha_1,\ldots,\alpha_k)}\sum_ET^E ev_{0*}^E(ev_1^{E*}\alpha_1\wedge \cdots \wedge ev_k^{E*}\alpha_k),
\end{multline*}
where
\[
\epsilon(\alpha_1,\ldots,\alpha_k) = \sum_{j = 1}^k j(|\alpha_j| + 1) + 1.
\]
The sum is well-defined by Gromov's compactness theorem~\cite{_Gromov:curves_}. The maps $\mu_k$ preserve grading by a virtual dimension calculation along the lines of~\cite[Sec. 3.7.5]{_FOOO_}.
It follows from the structure of the compactification $\bM(\I,E,\LL)$ and the properties of pull-back and push-forward that the operations $\mu_k$ satisfy the $A_\infty$ relations,
\begin{equation}\label{eq:air}
\sum_{\substack{k_1 + k_2 = k+1\\ 1 \leq q \leq k_1}} (-1)^{\star}\mu_{k_1}(\alpha_1,\ldots,\alpha_{q-1},\mu_{k_2}(\alpha_q,\ldots),\alpha_{q+k_2},\ldots,\alpha_k) = 0,
\end{equation}
where $\star = q-1 +\sum_{j = 1}^{q-1} |\alpha_j|.$ See~\cite{_STu_} for a detailed derivation including signs in the case of a single Lagrangian. The generalization to our setting is not hard.

Considering relation~\eqref{eq:air} for $k = 1$, we have
\[
\mu_1 \circ \mu_1(\alpha) = \mu_2(\mu_0,\alpha) + (-1)^{|\alpha| + 1} \mu_2(\alpha,\mu_0).
\]
So, if $\mu_0 = 0 \in CF(L_i,L_i)$ for $i \in A,$ then $\mu_1^2 = 0$, and we can define the \textbf{Floer cohomology} of $L_i,L_j,$ by
\[
HF^*(L_i,L_j) = H^*(CF^\bullet(L_i,L_j),\mu_1).
\]
In this case, we say the Lagrangians $L_i,\,i \in A,$ are \emph{tautologically unobstructed} with respect to $\Theta.$
Floer~\cite{_Floer_} introduced $HF^*(L_i,L_j)$ under the assumption $\pi_2(M,L_i) = 0$ for $i \in A.$ Then, a holomorphic disk with boundary in $L_i$ must have energy zero, so the stability condition implies $\bM(i,E,\LL) = \emptyset$ for $i \in A.$ It follows that the Lagrangians $L_i$ are tautologically unobstructed for all tame $\Theta.$ A similar case is when the ambient manifold~$M$ and the Lagrangians $L_i$ are exact. That is, there exist $\lambda \in A^1(M)$ with $d\lambda = \omega$ and $f_i : L_i \to \R$ with $\lambda|_{L_i} = df_i.$ Then, Stokes's theorem implies that a holomorphic disk with boundary in $L_i$ must have energy zero, so the stability condition implies $\bM(i,E,\LL) = \emptyset$ for $i \in A.$ Again, it follows that the Lagrangians $L_i$ are tautologically unobstructed for all tame $\Theta.$ This is the setting of~\cite{_Se_}.

To deal with general Lagrangian submanifolds, which may not be tautologically unobstructed, Fukaya-Oh-Ohta-Ono~\cite{_FOOO_} introduce the notion of \emph{bounding cochains}. A bounding cochain for $L_i$ is a cochain $b \in CF^1(L_i,L_i)$ with $\|b\| < 1$ that satisfies the Maurer-Cartan equation
\begin{equation}\label{eq:mcg}
\sum_{k=0}^\infty \mu_k(b^{\otimes k}) = 0.
\end{equation}
Let $b_i$ be a bounding cochain for $L_i.$ We define
\[
CF^*((L_i,b_i),(L_j,b_j)) = CF^*(L_i,L_j),
\]
and we define
\[
\hat\mu_k : \bigotimes_{j=1}^k CF^*((L_{i_{j-1}},b_{i_{j-1}}),(L_{i_j},b_{i_j})) \longrightarrow CF^*((L_{i_0},b_{i_0}),(L_{i_k},b_{i_k}))
\]
by
\begin{multline*}
\hat\mu_k(\alpha_1,\ldots,\alpha_k) =  \\
=\sum_{\sum_i m_i + k = l} \mu_l\left(b_{i_0}^{\otimes m_0}\otimes\alpha_1\otimes b_{i_1}^{\otimes m_1}\otimes \alpha_2\otimes \cdots\otimes\alpha_k\otimes b_{i_k}^{\otimes m_k}\right).
\end{multline*}
One verifies algebraically that equation~\eqref{eq:air} holds with $\mu_k$ replaced by $\hat\mu_k.$ Moreover, equation~\eqref{eq:mcg} is the same as $\hat \mu_0 = 0.$ Consequently, $\hat\mu_1^2 = 0.$ Thus, we define
\[
HF^*((L_{i_{j-1}},b_{i_{j-1}}),(L_{i_j},b_{i_j})) = H^*(CF^\bullet((L_{i_{j-1}},b_{i_{j-1}}),(L_{i_j},b_{i_j})),\hat\mu_1).
\]
Equation~\eqref{eq:air} for $k = 2$ implies that the composition map
\begin{multline*}
\circ :HF^*((L_{i_1},b_{i_1}),(L_{i_2},b_{i_2}))\otimes HF^*((L_{i_0},b_{i_0}),(L_{i_1},b_{i_1})) \to \\ \to HF^*((L_{i_0},b_{i_0}),(L_{i_2},b_{i_2}))
\end{multline*}
given by $[\alpha_2]\circ[\alpha_1] = (-1)^{|a_1|}[\hat \mu_2(\alpha_1,\alpha_2)]$ is well defined. This composition is associative by equation~\eqref{eq:air} for $k = 3.$ In particular, the Floer cohomology of a Lagrangian with itself $HF^*((L_{i_0},b_{i_0}),(L_{i_0},b_{i_0}))$ is an associative algebra.

A Hamiltonian flow $\phi_t : M \to M$ gives rise to a series of maps
\[
f_k^\phi : CF^*(L_{i_0},L_{i_1}) \otimesc \cdots \otimesc CF^*(L_{i_{k-1}},L_{i_k}) \to CF^*(\phi_1(L_{i_0}),\phi_1(L_{i_k}))
\]
defined using moduli spaces of holomorphic disks. The geometric construction is similar to~\cite[Section 4.6.1]{_FOOO_} and~\cite[Section 10e]{_Se_}.
The maps $f^\phi_k$ satisfy $\|f_k\| \leq 1$ with strict inequality for $k = 0$, and
\begin{multline}\label{eq:aif}
\sum_{\substack{l\\m_1+\cdots+m_l = k}}\mu_l(f_{m_1}^\phi(\alpha_1,\ldots,\alpha_{m_1}),\ldots,f^\phi_{m_l}(\alpha_{k-m_l+1},\ldots,\alpha_k)) = \\
=\sum_{\substack{k_1+k_2 = k+1\\1\leq i \leq k_1}} (-1)^\star f_{k_1}^\phi(\alpha_1,\ldots,\alpha_{i-1},\mu_{k_2}(\alpha_i,\ldots,\alpha_{i+k_2-1}),\ldots,\alpha_k),
\end{multline}
where $\star = q-1 +\sum_{j = 1}^{q-1} |\alpha_j|.$
One verifies algebraically that if $b$ solves the Maurer-Cartan equation~\eqref{eq:mcg}, then so does
\[
f^\phi_*(b) = \sum_k f^\phi_k(b^{\otimes k}).
\]

A fundamental property of Floer cohomology is that any Hamiltonian flow $\phi_t: M \to M$ gives rise~\cite[Sections 8c,8k]{_Se_} to a canonical element
\[
\iota_\phi \in HF^*((L_{i_0},b_{i_0}),(\phi_1(L_{i_0}),f^\phi_*(b_{i_0}))).
\]
If $\phi_t = \id_M$, then $\iota_\phi$ is the unit of the algebra $HF^*((L_{i_0},b_{i_0}),(L_{i_0},b_{i_0})).$ Moreover, if $\psi_t : M \to M$ is another Hamiltonian flow, then
\[
\iota_\psi \circ \iota_\phi = \iota_{\psi\circ\phi}.
\]
It follows that
\[
\iota_{\phi^{-1}} \circ \iota_\phi  = \id_{(L_{i_0},b_{i_0})}, \qquad \iota_\phi \circ \iota_{\phi^{-1}} = \id_{(\phi(L_{i_0}),f^\phi_*(b_{i_0}))}.
\]
In particular, composition with $\iota_\phi$ induces a canonical isomorphism
\[
HF^*((L_{i_0},b_{i_0}),(L_{i_1},b_{i_1})) \overset{\sim}{\longrightarrow} HF^*((L_{i_0},b_{i_0}),(\phi(L_{i_1}),f^\phi_*(b_{i_1}))).
\]
If $L,L',$ are tautologically unobstructed but do not intersect cleanly, we define
\[
HF^*(L,L') := HF^*((L,0),(\phi_1(L'),f^\phi_*(0)))
\]
where $\phi_t : M \to M$ is a Hamiltonian flow such that $\phi_1(L')$ is transverse to $L$. Since the choice of bounding chain is canonical, we omit it from the notation. By the preceding discussion, this definition does not depend on the choice of $\phi.$

\subsection{First applications}

Let $(M,I,J,K,g)$ be a hyperk\"ahler manifold, and let $L_i \subset M$ be $I$-holomorphic Lagrangians for $i$ belonging to a countable index set $A.$ In particular, the submanifolds $L_i$ are Lagrangian with respect to the symplectic form $\omega = \omega_J.$ In the following, all Floer-theoretic constructions will be carried out with respect to this symplectic form.

\begin{cy}\label{cy:end}\label{cy:tunobstructed}
For all $\Theta \in R_I\cap \J_\omega$ except a countable set, the $A_\infty$ operations $\mu_k : CF^*(L_i,L_i)^{\otimesc k} \to CF^*(L_i,L_i)$ are given by
\begin{equation}\label{eq:hkmuk}
\mu_k(\alpha_1,\ldots,\alpha_k)  =
\begin{cases}
d\alpha_1, & k = 1, \\
(-1)^{|\alpha_1|}\alpha_1 \wedge \alpha_2, & k = 2,\\
0, & k \neq 1, 2.
\end{cases}
\end{equation}
In particular, the Lagrangians $L_i$ are tautologically unobstructed.
\end{cy}
\begin{cy}\label{cy:coh}
For all $\Theta \in R_I\cap \J_\omega$ except a countable set, the Floer coboundary operator $\mu_1:CF^*(L_i,L_j) \to CF^*(L_i,L_j)$ coincides with the exterior derivative $d.$ Thus,
\[
HF^*(L_i,L_j) \simeq \bigoplus_{\substack{C \subset L_i\cap L_j \\\text{a component}}} H^*(C)[m_c]\otimes \Lambda.
\]
\end{cy}

\begin{proof}[Proof of Corollaries~\ref{cy:end} and~\ref{cy:coh}]
Constant maps have energy zero, so Theorem~\ref{tm:main} implies that the moduli spaces $\bM(\I,E,\LL)$ are empty unless $E = 0.$ On the other hand, for $\I \in A^{k+1}$ with $k \leq 1,$ the stability condition implies that $\bM(\I,E,\LL)$ is empty unless $E > 0.$ Consequently, we obtain equation~\eqref{eq:hkmuk} when $k = 0,1$ as well as Corollary~\ref{cy:coh}.

When $k \geq 2,$ constant maps can be stable. Let $\bM_{k+1}$ denote the moduli space of stable disks with $k+1$ cyclically ordered boundary marked points, up to biholomorphism. For $\I =  (i,i,\ldots,i)$ we have $\bM(\I,0,\LL) \simeq L_i \times \bM_{k+1}$, and $ev_j^0 : \bM(\I,0,\LL) \to L_i$ is the projection to the first factor for $j = 0,\ldots,k.$ Since $\bM(\I,0,\LL)$ is a smooth manifold of expected dimension, the virtual fundamental class coincides with the usual fundamental class.
We have
\begin{align*}
\mu_k(\alpha_1,\ldots,\alpha_k) &= (-1)^{\epsilon(\alpha_1,\ldots,\alpha_k)} ev^0_{0*}\left(\bigwedge_{j = 1}^k ev^{0*}_j\alpha_j\right)\\
& = (-1)^{\epsilon(\alpha_1,\ldots,\alpha_k)} ev^0_{0*}ev^{0*}_0\left(\bigwedge_{j = 1}^k\alpha_j\right) \\
& = (-1)^{\epsilon(\alpha_1,\ldots,\alpha_k)} \bigwedge_{j = 1}^k \alpha_j \wedge ev^0_{0*}(1).
\end{align*}
The fiber of $ev_0^0$ is zero dimensional only when $k = 2,$ and in this case $ev_0^0$ is the identity map. Equation~\eqref{eq:hkmuk} for $k \geq 2$ follows.
\end{proof}
\begin{rem}
Even without the assumption that $M$ is hyperk\"ahler and the $L_i$ are $I$-holomorphic, the $A_\infty$ algebra $(CF(L_i,L_i),\mu_k)$ is a deformation of the differential graded algebra of differential forms on $L_i$. See Theorem X of~\cite{_FOOO_} and Theorem 3 of~\cite{_STu_}. Corollary~\ref{cy:end} says the deformation is trivial if $L_i$ is $I$-holomorphic and $M$ is hyperk\"ahler.
\end{rem}

\subsection{Algebraic framework}\label{ssec:af}
Before presenting further applications, we recall some definitions pertaining to abstract $A_\infty$ categories.
In general, a {\bf curved $A_\infty$ category} $\A$ consists of the following data:
\begin{itemize}
\item
A collection of objects $\ob \A.$
\item
For each pair of objects $A,B \in \ob\A,$ a graded normed $\Lambda$ vector space $Hom_\A(A,B).$
\item
For each $(k+1)$-tuple of objects $A_0,\ldots,A_k,\, k \geq 0,$ a multilinear map
\[
\mu_k^\A : \bigotimes_{i = 1}^k Hom_\A(A_{i-1},A_i) \to Hom_\A(A_0,A_k)[2-k].
\]
\end{itemize}
The maps $\mu_k^\A$ must satisfy $\|\mu_k^\A\| \leq 1$ with strict inequality for $k = 0$, as well as  the $A_\infty$ relations~\eqref{eq:air}.
For example, we can take the objects to be Lagrangian submanifolds, define $Hom(L_0,L_1) = CF^*(L_0,L_1)$, and define the maps $\mu_k^\A$ as in Section~\ref{ssec:gf}. An {\bf $A_\infty$ category} is a curved $A_\infty$ category with $\mu_0 = 0.$

A curved $A_\infty$ \textbf{functor} $f: \A \to \B$ consists of a map $f : \ob \A \to \ob \B$ along with multilinear maps
\[
f_k : \bigotimes_{i = 1}^k Hom_\A(A_{i-1},A_i) \to Hom_\B(f(A_0),f(A_k))[1-k]
\]
for each $k+1$ tuple of objects $A_0,\ldots,A_k \in \ob \A,$ for $k \geq 0.$ The maps $f_k$ must satisfy $\|f_k\| \leq 1$ with strict inequality for $k = 0$, as well as the relation
\begin{multline*}
\sum_{\substack{l\\m_1+\cdots+m_l = k}}\mu_l^\B(f_{m_1}(\alpha_1,\ldots,\alpha_{m_1}),\ldots,f_{m_l}(\alpha_{k-m_l+1},\ldots,\alpha_k)) = \\
=\sum_{\substack{k_1+k_2 = k+1\\1\leq i \leq k_1}} (-1)^\star f_{k_1}(\alpha_1,\ldots,\alpha_{i-1},\mu_{k_2}^\A(\alpha_i,\ldots,\alpha_{i+k_2-1}),\ldots,\alpha_k),
\end{multline*}
where $\star = q-1 +\sum_{j = 1}^{q-1} |\alpha_j|.$ The sum on the left-hand side converges because $\|f_0\| < 1.$ An $A_\infty$ functor is a curved $A_\infty$ functor with $f_0 = 0.$

Let $\A$ be a curved $A_\infty$ category and $A \in \ob\A.$ A Maurer-Cartan element or bounding cochain~\cite{_FOOO_} for $A$ is an element $b \in Hom_\A(A,A)$ with $\|b\| < 1$ such that
\[
\sum_{k = 0}^\infty \mu_k(b^{\otimes k}) = 0.
\]
Let $f : \A \to \B$ be a curved $A_\infty$ functor. If $b$ is a bounding cochain for $A \in \ob\A$, then
\[
f_*(b) = \sum_{k=0}^\infty f_k(b^{\otimes k})
\]
is a bounding cochain for $f(A) \in \ob \B.$

To a curved $A_\infty$ category $\A,$ we associate the $A_\infty$ category $\widehat \A$ defined as follows. An object of $\widehat \A$ is a pair of an object $A \in \ob\A$ and a bounding cochain $b$ for $A.$
For $(A_i,b_i) \in \ob\widehat\A$ we define
\[
Hom_{\widehat \A}((A_i,b_i),(A_j,b_j)) = Hom(A_i,A_j).
\]
For $k \geq 1$ and $\alpha_i \in Hom_{\widehat\A}((A_{i-1},b_{i-1}),(A_i,b_i))$, we define
\[
\mu_k^{\widehat\A}(\alpha_1,\ldots,\alpha_k) = \sum_{\sum_i m_i + k = l} \mu_l^\A(b_0^{\otimes m_0}\otimes\alpha_1\otimes b_1^{\otimes m_1}\otimes \alpha_2\otimes \cdots\otimes\alpha_k\otimes b_k^{\otimes m_k}).
\]
To a curved $\A_\infty$ functor $f:\A \to \B$, we associate the $A_\infty$ functor $\hat f : \widehat \A \to \widehat \B$  defined as follows. The map $\hat f : \ob\widehat\A \to \ob\widehat\B$ is given by
\[
\hat f((A,b)) = (f(A),f_*(b)).
\]
For $k \geq 1$ and $\alpha_i \in Hom_{\widehat\A}((A_{i-1},b_{i-1}),(A_i,b_i))$, we define
\[
\hat f_k(\alpha_1,\ldots,\alpha_k) = \sum_{\sum_i m_i + k = l} f_l(b_0^{m_0}\otimes\alpha_1 \otimes \cdots\otimes \alpha_k \otimes b_k^{\otimes m_k}).
\]

Let $\A$ be an $A_\infty$ category. The associated cohomological category $H(\A)$ has the same objects, its morphism spaces are given by
\[
Hom_{H(\A)}(A,B) = H^*(Hom_\A(A,B),\mu_1^\A),
\]
and the composition of morphisms is given by
\begin{equation}\label{eq:comp}
[a_2]\circ [a_1]  = (-1)^{|a_1|}[\mu_2(a_1,a_2)].
\end{equation}
Composition of morphisms is associative because of relation~\eqref{eq:air} for $k = 3.$ If we drop the sign in equation~\eqref{eq:comp}, composition is no longer associative. Rather, we obtain an $A_\infty$ category, called the cohomological $A_\infty$ category, which has all operations zero except for $k = 2$. An $A_\infty$ functor $f: \A \to \B$ induces a functor $H(f): H(\A) \to H(B).$ Suppose $H(\A),H(\B),$ have identity morphisms. Then the functor $f$ is called a quasi-equivalence if $H(f)$ is an equivalence. An $A_\infty$ category is called \textbf{formal} if it is $A_\infty$ equivalent to its cohomological $A_\infty$ category.

A paradigmatic example of quasi-equivalent $A_\infty$ categories is the following. Let $(M,\omega)$ be as in Section~\ref{ssec:gf} and let $\LL$ be a collection of graded spin compact Lagrangian submanifolds $L_i \subset M, \, i \in A,$ that intersect cleanly. For $j = 0,1,$ let $\A_j$ denote the curved $A_\infty$ category associated to $\LL$ using the almost complex structure $\Theta_j \in \J_\omega.$ Since $\J_\omega$ is contractible, one can always find a path $\Theta_t \in \J_\omega,\, t \in [0,1],$ from $\Theta_0$ to $\Theta_1.$ To such a path one can associate a curved $A_\infty$ functor $f^\Theta : \A_0 \to \A_1$ such that the $A_\infty$ functor $\hat f^\Theta : \widehat \A_0 \to \widehat \A_1$ is a quasi-equivalence. Thus, one can associate an $A_\infty$ category $\widehat \A_\LL$ to $\LL$ that, up to quasi-equivalence, depends only on the symplectic form $\omega.$ We call $\widehat \A_\LL$ the \textbf{Fukaya category} of $\LL.$

Similarly, let $\phi_{i,t} : M \to M, \, i \in A,\, t \in [0,1],$ be a collection of Hamiltonian flows such that the Lagrangian submanifolds $\phi_{i,1}(L_i)$ intersect cleanly. Let $\LL^\phi$ denote the collection of Lagrangians $\phi_{i,1}(L_i), i\in A.$ Then we have a curved $A_\infty$ functor $f^{\phi} : \A_\LL \to \A_{\LL^\phi}$ such that the $A_\infty$ functor $\hat f^\phi : \widehat \A_\LL\to \widehat \A_{\LL^\phi}$ is a quasi-equivalence. Thus, one sees that up to quasi-equivalence, the Fukaya category $\widehat \A_\LL$ depends only on the Hamiltonian isotopy classes of the Lagrangian submanifolds $L_i, i \in A.$

It follows that the Fukaya category $\widehat \A_\LL$ is well-defined even when the Lagrangians $L_i$ do not intersect cleanly. Indeed, it is always possible to choose $\phi_{i,t}$ such that the Lagrangians $\phi_{i,1}(L_i)$ intersect cleanly, so we may define $\widehat \A_\LL := \widehat \A_{\LL^\phi}.$ Any two choices of Hamiltonian flows $\phi_{i,t}$ differ by a Hamiltonian flow, so the preceding discussion shows that, up to quasi-equivalence, the definition does not depend on the choice of $\phi_{i,t}.$

\subsection{Locality}\label{ssec:loc}
We return to the setting where $(M,I,J,K,g)$ is a hyperk\"ahler manifold, $\omega = \omega_J$ and $\LL$ is a collection of graded spin compact $I$-holomorphic Lagrangians $L_i \subset M,\, i \in A,$ with $A$ finite now. We do not assume the Lagrangians $L_i$ intersect cleanly. Let $\phi_{i,t} : M \to M, \, i \in A, t \in [0,1],$ be a collection of real analytic Hamiltonian flows such that the Lagrangian submanifolds $\phi_{i,t}(L_i)$ intersect cleanly for $t \in (0,1].$ Let $\LL_t$ denote the collection of Lagrangian submanifolds $\phi_{i,t}(L_i), i \in A.$ The following is a special case of Theorem~\ref{tm:loc}.
\begin{cy}\label{cy:loc}
For all $\Theta \in R_I \cap \J_\omega$ except a countable set, we have the following. Choose $E_0 \geq 0,K \in \Z_{\geq 0},$ and an open $V \subset M$ such that
\[
\bigcup_{i \in A} L_{i} \subset V.
\]
Then, there exists $\epsilon > 0$ such that for all
\begin{gather*}
k \in \Z \cap [0,K], \qquad \I \subset A^{k+1},\\
E \in [0, E_0], \qquad t \in [0,\epsilon], \qquad (u,\z) \in \widetilde\M(\I,E,\LL_t),
\end{gather*}
we have
\[
u(D\setminus \zz) \subset V.
\]
\end{cy}
As a consequence of Corollary~\ref{cy:loc} and the discussion in Section~\ref{ssec:af}, we obtain the following locality statement for $\widetilde \A_\LL.$
\begin{cy}\label{cy:loccat}
Choose $E_0,K$ and $V,$ as in Corollary~\ref{cy:loc}. Then $\widehat\A_\LL$ is quasi-equivalent to an $A_\infty$ category $\widehat\B$ with the same objects, such that there exist operations
\[
\mu_k^{E_0},\mu_k^V : \bigotimes_{j = 1}^k Hom_{\widehat\B}(L_{i_{j-1}},L_{i_j}) \to Hom_{\widehat\B}(L_{i_0},L_{i_k})[2-k].,
\]
with
\[
\mu_k^{\widehat\B} = \mu_k^{E_0} + \mu_k^V, \qquad \| \mu_k^{E_0}\| < e^{-E_0},
\]
and $\mu_k^V$ only depends on the geometry of $V.$ In fact, we may take $\widehat \B = \widehat \A_{\LL_t}$ with $t < \epsilon$ and $\epsilon$ as in Corollary~\ref{cy:loc}.
\end{cy}
\begin{rem}\label{rem:loccat}
Starting from Corollary~\ref{cy:loccat}, it should be possible to show that in fact, up to quasi-equivalence, the $A_\infty$ category $\widehat A_\LL$ depends only on $V.$ The proof would use a categorical generalization of the obstruction theory of $A_{n,K}$ algebras developed in Section 7.2.6 of~\cite{_FOOO_}. Indeed, by Corollary~\ref{cy:loccat}, for arbitrarily large $E_0$ and $K,$ we can define an $A_{E_0,K}$ category $\widehat \A^V_{\LL_t}$ depending only on $V$ by replacing the usual structure maps of $\widehat \A_{\LL_t}$ with the maps $\mu_k^V.$ A generalization of Theorem~\ref{tm:loc} should imply that for arbitrarily large $E_0,K,$ there exists $\epsilon > 0$ such that for $t,t' < \epsilon$ a truncated version of the usual quasi-equivalence $\widehat \A_{\LL_t} \to \widehat \A_{\LL_{t'}}$ gives an $A_{E_0,K}$ quasi-equivalence $\widehat \A^V_{\LL_t} \to \widehat \A^V_{\LL_{t'}}$. Then one can apply the categorical version of Theorem 7.2.72 of~\cite{_FOOO_} to construct by a limiting procedure an $A_\infty$ category $\widehat \A^V_{\LL}$ depending only on $V$ by extending the $A_{n,K}$ structure on $\widehat \A^V_{\LL_{t_0}}$ for a fixed $t_0.$ Examining the construction, one finds that $\widehat \A^V_{\LL}$ is quasi-equivalent to $\widehat \A_{\LL}.$
\end{rem}

\begin{rem}\label{rem:st}
Kapustin gives physical arguments~\cite{_K1_,_K2_} for the Fukaya category of a hyperk\"ahler manifold $(M,I,J,K,g)$ equipped with the symplectic form $\omega_J$ to coincide with the category of deformation quantization modules~\cite{_KaS12_} on the holomorphic symplectic manifold $(M,I,\Omega_I).$ Given a spin holomorphic Lagrangian submanifold $L \subset M,$ D'Agnolo-Schapira~\cite{_DaS_} construct a deformation quantization module $D^\bullet_L$ supported on $L.$ Thus, one would expect that $\widehat \A_{\LL}$ is equivalent in some sense to the full subcategory $\D_\LL$ of deformation quantization modules on $M$ with objects $D^\bullet_{L_i}$ for $i \in A.$ By definition, the category $\D_\LL$ depends only on an arbitrarily small open neighborhood $V$ of $\cup_{i \in A}L_i.$ Corollary~\ref{cy:loccat} and Remark~\ref{rem:loccat} show that although a priori the Fukaya category $\widehat \A_\LL$ depends on the symplectic geometry of the whole manifold $M,$ in fact, $\widehat \A_\LL$ also only depends on $V.$ This can be seen as a first step toward proving the expected equivalence of $\widehat \A_\LL$ and $\D_\LL.$

In a similar vein, Behrend-Fantechi~\cite{_BF_} outline the construction of a differential graded category $\B$ associated to a collection of holomorphic Lagrangians in holomorphic symplectic manifold $(M,I,\Omega_I)$ and speculate on its relation to the corresponding Fukaya category. Again, $\B$ depends only on an arbitrarily small neighborhood of the Lagrangians. The cohomology of the morphism complex $Hom_\B(L_0,L_1)$ is called the virtual de Rham cohomology of the intersection $L_0 \cap L_1.$ If $L_0$ and $L_1$ intersect cleanly, the virtual de Rham cohomology of $L_0 \cap L_1$ coincides with the usual de Rham cohomology. When $M$ is hyperk\"ahler, Corollary~\ref{cy:coh} asserts the same for the Floer cohomology $HF^*(L_0,L_1)$ up to tensoring with $\Lambda$.

Brav-Bussi-Dupont-Joyce-Szendroi~\cite{_BBDJS_} and Bussi~\cite{_Bu_} give a construction of the virtual de Rham cohomology of a pair of Lagrangians $L_0,L_1 \subset M$ in terms of the hypercohomology of a perverse sheaf $P^\bullet_{L_0,L_1}$ supported on $L_0 \cap L_1.$ As explained in Remark 6.15 of~\cite{_BBDJS_}, it follows from the work of Kashiwara-Schapira~\cite{_KaS08_} that the perverse sheaf $P_{L_0,L_1}$ can be recovered as $R\mathcal{H}om(D^\bullet_{L_0},D^\bullet_{L_1}).$ The same remark speculates on a connection between $\mathbb{H}^*(P^\bullet_{L_0,L_1})$ and $HF^*(L_0,L_1).$
\end{rem}

\subsection{Formality}
Let $(M,I,J,K,g)$ be a hyperk\"ahler manifold, take $\omega = \omega_J$ and let $\LL$ be a collection of graded spin compact $I$-holomorphic Lagrangians $L_i \subset M,\, i \in A,$ with $A$ countable. By Corollary~\ref{cy:tunobstructed}, we may choose $\Theta \in R_I\cap \J_\omega$ with respect to which $L_i$ is tautologically unobstructed for $i \in A.$

\begin{cy} \label{cy:formal}(I. Smith)
The $A_\infty$ algebra $(CF^*(L_i,L_i),\mu_k)$ is formal.
\end{cy}

\begin{proof}
Corollary~\ref{cy:end} shows that $(CF^*(L_i,L_i),\mu_k)$ is the usual dg algebra of differential forms on $L_i$ with the product modified by a sign as in equation~\eqref{eq:comp}. Moreover, $\omega_I$ is a K\"ahler form for $L_i$. So, we apply the formality theorem of Deligne-Griffiths-Morgan-Sullivan~\cite{_DGMS_}.
\end{proof}

The following question, suggested by I. Smith, is natural in light of Corollary~\ref{cy:formal}. Let $\widehat \A_\LL$ denote the Fukaya category of $\LL$ as in Section~\ref{ssec:af}.

\begin{qn}\label{q:formal}
Is $\widehat\A_\LL$ formal?
\end{qn}

Behrend-Fantechi~\cite[Conjecture 5.8]{_BF_} formulate a similar question concerning the degeneration of the spectral sequence computing the virtual de Rham cohomology of the intersection of two holomorphic Lagrangians.

Abouzaid-Smith~\cite{_AS_} prove a formality theorem for $\widehat \A_\LL$ when $M$ belongs to a certain family of Nakajima quiver varieties and $\LL$ is the collection of Lagrangians introduced in~\cite{_SS_}. Their proof does not make use of the hyperk\"ahler structure on $M.$ It would be interesting to find an alternative proof that does.


\section{Special Lagrangian submanifolds}\label{sec:slag}


For the benefit of the reader, we recall some basic facts about special Lagrangian submanifolds. We follow \cite{_Harvey_Lawson:Calibrated_}.

\begin{df}
A {\bf Calabi-Yau} manifold is a K\"ahler manifold $(M,J,\omega)$
of complex dimension $n$ along with a holomorphic
$(n,0)$ form $\Phi$ such that
\begin{equation}\label{eq:cy}
\omega^n/n! = (-1)^{n(n-1)/2}(\sqrt{-1}/2)^n\Phi \wedge \overline\Phi.
\end{equation}
An $n$-dimensional submanifold $L \subset M$ is {\bf special Lagrangian} if
\[
\omega|_L = 0, \qquad \Im \Phi|_L = 0.
\]
\end{df}
The following lemmas concern a hyperk\"ahler manifold $(M,I,J,K,g)$ of complex dimension $2n.$ For complex structure $\Theta \in R_I,$ write $\omega_\Theta$ for the corresponding K\"ahler form. Observe that if $\Theta \in R_I,$ then also $I\Theta \in R_I.$
\begin{lm}\label{lm:hkcy}
Let $\Theta \in R_I.$ Then the $\Theta$-holomorphic $(2n,0)$ form
\[
\Phi_\Theta = (\omega_I - \sqrt{-1}\omega_{I\Theta})^n/n!
\]
makes the K\"ahler manifold $(M,\Theta,\omega_\Theta)$ into a Calabi-Yau manifold.
\end{lm}
\begin{proof}
We consider first a special case. Let $M = \C^{2n}$ with coordinates $z_i,w_i,\, i = 1,\ldots,n.$ Let $g$ be the standard Euclidean metric. Take $\Theta$ to be the standard complex structure, so
\[
\omega_\Theta = \frac{\sqrt{-1}}{2} \sum_{i = 1}^n dz_i \wedge d\bar z_i + dw_i \wedge d\bar w_i.
\]
A $\Theta$-holomorphic symplectic form is given by
\[
\Omega_\Theta = \sum_{i = 1}^n dz_i \wedge dw_i.
\]
Let $I$ be the complex structure such that
$\omega_I = \Re \Omega_\Theta.$ It follows that $\omega_{I\Theta} = - \Im \Omega_\Theta.$ Thus, $\Phi_\Theta = \Omega_\Theta^n/n!,$ and equation~\eqref{eq:cy} with $2n$ in place of $n$ follows by a straightforward calculation.

To deal with a general hyperk\"ahler manifold, it suffices to verify equation~\eqref{eq:cy} at each point. At a given point, one can choose coordinates $z_i,w_i,$ such that $\omega_I,\omega_\Theta,\omega_{I\Theta},$ are given by the same expressions as in the preceding special case. Indeed, the triple $I,J := \Theta, K:=I\Theta$ satisfies the standard quaternionic relations $IJ = -JI = K,$ and $I^2 = J^2 = K^2 = -1.$
\end{proof}

\begin{lm}[{\cite[p. 154]{_Harvey_Lawson:Calibrated_}}]\label{lm:hlsl}
Let $\Omega_I$ be as in Remark~\ref{rem:hkhs}, let $\Theta \in R_I$ and let $\Phi_\Theta$ be as in Lemma~\ref{lm:hkcy}. Suppose $L$ is a holomorphic Lagrangian submanifold of the holomorphic symplectic manifold $(M,I,\Omega_I).$ Then $L$ is a special Lagrangian submanifold of the Calabi-Yau $(M,\Theta,\omega_\Theta,\Phi_\Theta).$
\end{lm}
\begin{proof}
Since $\Omega_I|_L = 0,$ we have $\omega_J|_L = 0$ and $\omega_K|_L = 0.$ It follows that $\omega_\Lambda|_L = 0$ for any $\Lambda \in R_I.$ In particular, $\omega_\Theta|_L  = 0.$ Moreover, since $\omega_{I\Theta}|_L = 0,$ we have $\Im \Phi_\Theta|_L = \Im \omega_I^n|_L/n! = 0$ since $\omega_I$ is real.
\end{proof}

Lemma~\ref{lm:hlsl} is one of the few known approaches to constructing special Lagrangian submanifolds. Another approach produces special Lagrangian submanifolds as the fixed points of anti-holomorphic isometric involutions~\cite{_Bry_}. In the non-compact setting, Lie group actions can be used to produce special Lagrangians~\cite{_Harvey_Lawson:Calibrated_}. Once a single special Lagrangian $L$ has been constructed, deformation theory can be used~\cite{_McL_} to constructed a family of nearby special Lagrangians modelled on $H^1(L).$ In total, these techniques are only able to produce special Lagrangians in a handful of special situations.

On the other hand, ignoring instanton corrections, special Lagrangians are expected to be quite plentiful. Indeed, the existence of a special Lagrangian representative of an isomorphism class of objects in the Fukaya category ought to be roughly equivalent to stability in the sense of Douglas and Bridgeland~\cite{_Bri_,_Do_,_TY_}. Thus there should be enough special Lagrangians for every object of the Fukaya category to admit a Harder-Narasimhan filtration.

To explain the apparent dearth of special Lagrangians, the first author and Tian~\cite{_ST_} posit that instanton corrections cannot be ignored. For a Lagrangian $L$ together with a bounding chain $b,$ they propose an explicit instanton corrected special Lagrangian equation. If $L$ is tautologically unobstructed and $b = 0,$ the instanton-corrected special Lagrangian equation reduces to the usual special Lagrangian equation.

From the point of view of~\cite{_ST_}, Corollary~\ref{cy:tunobstructed} explains why Lemma~\ref{lm:hlsl} can be true. Namely, the hyperk\"ahler structure suppresses the instanton corrections which would in general prevent the existence of ordinary special Lagrangians. Similarly, fixed points of anti-holomorphic isometric involutions are known to be tautologically unobstructed~\cite{_FO3i_}.

\section{Local symplectic action}
\label{sec:lsa}
Let $(M,\omega)$ be a real analytic symplectic manifold, and let \[
L_0,L_1 \subset M
\]
be real analytic Lagrangian submanifolds. Let $\Theta$ be an $\omega$-tame almost complex structure, not necessarily real analytic, and let $g = g_\Theta$ denote the associated Riemannian metric. Assume that $(M,g)$ is complete and there exists a compact set $N \subset M$ such that
\[
\dist_g(L_0\setminus N, \, L_1\setminus N) > 0.
\]
Write
\[
H = \{x + \sqrt{-1}y\,|y \geq 0, \; x^2 + y^2 \leq 1\} \subset \C
\]
for the intersection of the closed upper half-plane and the closed unit disk. We think of the intervals $[-1,0]$ and $[0,1]$ as subsets of $\R$ and thus as subsets of $H \subset \C.$ Let $B_r(p)\subset M$ denote the open ball of radius $r$ centered at $p$ with respect to the metric $g.$ For a path $\gamma : [a,b] \to M,$ let
\[
\ell_g(\gamma) = \int_a^b |\dot \gamma(t)|_g dt
\]
denote the length of $\gamma$ with respect to $g.$ For $W,V,$ manifolds with corners, a continuous map $f: W \to V$ is called {\bf piecewise smooth} if we can find a smooth triangulation of $W$ such that $f$ restricted to any simplex is smooth.
The main result of this section is the following theorem.

\begin{tm}\label{tm:vg}
There exists $\epsilon_0 > 0$ such that for each $\epsilon \in (0,\epsilon_0]$ there exists a $\delta_0 = \delta_0(\epsilon) > 0$ with the following significance.
Suppose
\[
\gamma : [0,1] \to M
\]
with $\gamma(i) \in L_i$ for $i = 0,1,$ and either $\ell_g(\gamma) < \delta_0$ or there exists a point
$
p \in L_0\cap L_1
$
with $\gamma([0,1]) \subset B_{\epsilon/9}(p).$
\begin{enumerate}
\item\label{it:vg}
There exists a piecewise smooth $v_\gamma : H\to B_\epsilon(\gamma(0))$ with
\begin{gather*}
v_\gamma([-1,0]) \subset L_1, \qquad v_\gamma([0,1]) \subset L_0,\\
v_\gamma(e^{\sqrt{-1}\pi t}) = \gamma(t).
\end{gather*}
\item\label{it:inv}
$\int_H v_\gamma^*\omega$ depends only on $\gamma$ and not on the choice of $v_\gamma.$
\end{enumerate}
\end{tm}
Thus, we define the {\bf local symplectic action} of $\gamma$ to be
\[
a(\gamma) = \int_H v_\gamma^*\omega,
\]
where $v_\gamma$ is as in~\ref{it:vg}.

For the proof of Theorem~\ref{tm:vg}, we use the following theorem on the Whitney regularity of real analytic sets~\cite{_Bi_,_Ha_}.
\begin{tm}\label{tm:wr}
Let $X$ be a compact connected analytic subset of $\R^n.$ Then there is a positive integer $\nu_X$ and a constant $C_X$ such that any two points $x,y \in X$ can be joined by a semi-analytic curve $\gamma$ in $X$ of length
\[
\ell(\gamma) \leq C_X|x-y|^{\frac{1}{\nu_X}}.
\]
\end{tm}
The proof of Theorem~\ref{tm:wr} uses the \L ojasiewicz inequality and local resolution of singularities.
A concise exposition can be found in~\cite[Theorem 6.1]{_BM88_}. Using a covering argument, we obtain the following corollary.
\begin{cy}\label{cy:cs}
There is a positive integer $\nu$ and a constant $C$ such that any two points $p,q,$ in the same connected component of $L_0\cap L_1$ can be joined by a piecewise smooth curve $\alpha: [0,1] \to L_0\cap L_1 \subset M$ with
\[
\ell_g(\alpha) \leq C\dist_g(p,q)^{\frac{1}{\nu}}.
\]
\end{cy}
We begin by proving the following lemma.
\begin{lm}\label{lm:tri}
Let $(X,h)$ be a complete Riemannian manifold, and let $Y \subset X$ be a subset such that injectivity radius of $X$ at all points of $Y$ is bounded below by $\delta.$ Suppose three paths $\chi,\beta_0,\beta_1 : [0,1] \to Y$ form a triangle,
\[
\chi(0) = \beta_0(1), \qquad \chi(1) = \beta_1(1), \qquad \beta_0(0) = \beta_1(0),
\]
with $\chi$ piecewise smooth and $\beta_0,\beta_1,$ smooth.
Moreover, suppose
\[
\chi([0,1]), \beta_0([0,1]), \beta_1([0,1]) \subset B_{\delta/3}(\chi(0)).
\]
Then there exists a piecewise smooth map $\zeta: H \to B_\delta(\chi(0))\subset X$ such that for $x\in [0,1]$ we have
\begin{gather*}
\zeta(x) = \beta_0(x), \qquad \zeta(-x) = \beta_1(x), \\
\zeta(e^{\sqrt{-1}\pi x}) = \chi(x).
\end{gather*}
\end{lm}
\begin{proof}
For $z = x + \sqrt{-1} y \in H,$ write
\begin{gather*}
u(z) = \frac{2 x (y+1)}{x^2 + (y+1)^2}, \qquad v(z) = \frac{(y+1)^2 - x^2}{x^2 +(y+1)^2}, \\
w(z) = u(z) + \sqrt{-1}v(z).
\end{gather*}
So, $w(z)$ is the stereographic projection of $z$ from $-\sqrt{-1}$ to the unit circle.
Write
\[
t(z) = \frac{1}{\pi}\arg{w(z)}.
\]
For $i = 0,1,$ and $x \in [0,1],$ define $\xi_i(x) \in T_{\beta_i(x)}M$ by
\[
|\xi_i(x)| <\frac{2\delta}{3}, \qquad\exp_{\beta_i(x)}{\xi_i(x)} =
\begin{cases}
\chi(t(x)), & i = 0, \\
\chi(t(-x)), & i = 1.
\end{cases}
\]
Such $\xi_i$ exist because by the triangle inequality, $\dist_g(\beta_i(s),\chi(t)) < \frac{2\delta}{3}$ for all $s,t.$ Write
\[
\varpi(z) = \frac{x}{y+1},
\]
for the stereographic projection of $z$ from $-\sqrt{-1}$ to $\R.$ Take
\[
\zeta(z) =
\begin{cases}
\exp_{\beta_0(\varpi(z))}\left(\frac{y}{v(z)}\xi_0(\varpi(z))\right), & x \geq 0,\\
\exp_{\beta_1(-\varpi(z))}\left(\frac{y}{v(z)}\xi_1(-\varpi(z))\right), &  x < 0.
\end{cases}
\]
Finally, for $i = 0$ or $1$ depending on whether $x\geq 0$ or $x < 0$, we have
\begin{align*}
\dist_g(\zeta(z),\chi(0))
&\leq \dist_g(\zeta(z),\zeta(w(z))) + \dist_g(\zeta(w(z)),\chi(0)) \\
&\leq |\xi_i((-1)^i\varpi(z))| + \dist_g(\chi(t(z)),\chi(0))  \\
&< \frac{2\delta}{3} + \frac{\delta}{3} \\
&= \delta.
\end{align*}
\end{proof}

For $r > 0,$ let $U_r$ denote the open $r$ neighborhood of $L_0 \cap L_1$ with respect to $g.$ If necessary replacing $N$ with $N \cup \overline{U_1},$ we may assume that $U_1 \subset N.$ Let $g_{L_i}$ denote the induced metric on $L_i$ for $i = 0,1.$ Let $\epsilon_N$ denote the minimum injectivity radius of $(M,g)$ that occurs at a point $q \in N$ and let $\epsilon_{L_i}$ denote the minimum injectivity radius of $(L_i,g_{L_i})$ that occurs at a point $q \in N \cap L_i.$ Choose
\begin{equation}\label{eq:e0}
\epsilon_1 \leq \min\left(\epsilon_N,\epsilon_{L_0},\epsilon_{L_1},\frac{1}{2}\right).
\end{equation}

Let $B_r(q,L_i)\subset L_i$ denote the ball of radius $r$ centered at $q$ with respect to the metric $g_{L_i}.$ After possibly making $\epsilon_1$ smaller, we may assume that for every $q\in N\cap L_i$ we have
\begin{equation}\label{eq:LM}
B_{r/2}(q) \cap L_i \subset B_{r}(q,L_i) \qquad \forall r < \epsilon_1.
\end{equation}

\begin{cy}\label{cy:trip}
Let $\epsilon \leq \epsilon_1.$ Let $\gamma:[0,1] \to M$ with $\gamma(i) \in L_i$ for $i =0,1.$ Suppose $p \in L_0 \cap L_1$ such that
\[
\gamma([0,1]) \subset B_{\epsilon/9}(p).
\]
Then there exists $v_\gamma$ as in Theorem~\ref{tm:vg}\ref{it:vg} with $v_\gamma(0) = p.$
\end{cy}
\begin{proof}
For $i = 0,1,$ let $\beta_i : [0,1] \to L_i$ be a path with $\beta_i(0) = p$ and $\beta_i(1) = \gamma(i)$ and $\beta_i([0,1]) \subset B_{\frac{2\epsilon}{9}}(p).$ Such $\beta_i$ exist because by~\eqref{eq:LM} with $r = \frac{2\epsilon}{9},$ we have $\gamma(i) \subset B_{\frac{2\epsilon}{9}}(p,L_i),$ and $\frac{2\epsilon}{9}$ is less than the injectivity radius of $L_i.$ It follows that $\beta_i([0,1]) \subset B_{\frac{\epsilon}{3}}(\gamma(0)).$ Apply Lemma~\ref{lm:tri} with $\chi = \gamma.$
\end{proof}
\begin{lm}\label{lm:gbp}
For $r > 0$ there exists $\delta(r) > 0$ with the following significance.
If $\gamma :[0,1] \to M$ with $\gamma(i) \in L_i$ for $i = 0,1,$ and $\ell_g(\gamma)< \delta(r),$ then there exists $p \in L_0 \cap L_1$ such that
\begin{equation}\label{eq:gbp}
\gamma([0,1]) \subset B_{r}(p).
\end{equation}
\end{lm}
\begin{proof}
Take
\[
\delta(r) = \min\{r/2, \dist_g(L_0 \setminus U_{r/2} , L_1 \setminus U_{r/2})\}.
\]
Since
\[
\ell_g(\gamma) < \dist_g(L_0 \setminus U_{r/2} , L_1 \setminus U_{r/2}),
\]
either $\gamma(0) \in U_{r/2}$ or $\gamma(1) \in U_{r/2}.$ So, there exists $p \in L_0 \cap L_1$ such that either $\gamma(0) \in B_{r/2}(p)$ or $\gamma(1) \in B_{r/2}(p).$ Since $\ell_g(\gamma) < r/2,$ inclusion~\eqref{eq:gbp} follows.
\end{proof}

Theorem~\ref{tm:vg} is an immediate consequence of the following more precise statement.
\begin{tm}\label{tm:prec}
Let $\epsilon_0 > 0$ satisfy
\begin{align}\label{eq:e1}
\epsilon_0 &\leq \frac{\epsilon_1}{18}, \\
\label{eq:cc}
\epsilon_0 &< \frac{1}{2}\min\{\dist_g(Z,W)| Z,W, \text{components of $L_0\cap L_1$}\},
\end{align}
and with $\nu,C,$ as in Corollary~\ref{cy:cs},
\begin{equation}\label{eq:wr}
C(2\epsilon_0)^{\frac{1}{\nu}} < \frac{\epsilon_1}{9}.
\end{equation}
For $\epsilon \in (0,\epsilon_0],$ with $\delta$ as in Lemma~\ref{lm:gbp}, take
\begin{equation}\label{eq:d0}
\delta_0 = \delta\left(\frac{\epsilon}{9}\right).
\end{equation}
Then Theorem~\ref{tm:vg} holds.
\end{tm}
\begin{figure}[ht]
\centering
\includegraphics[width=12cm]{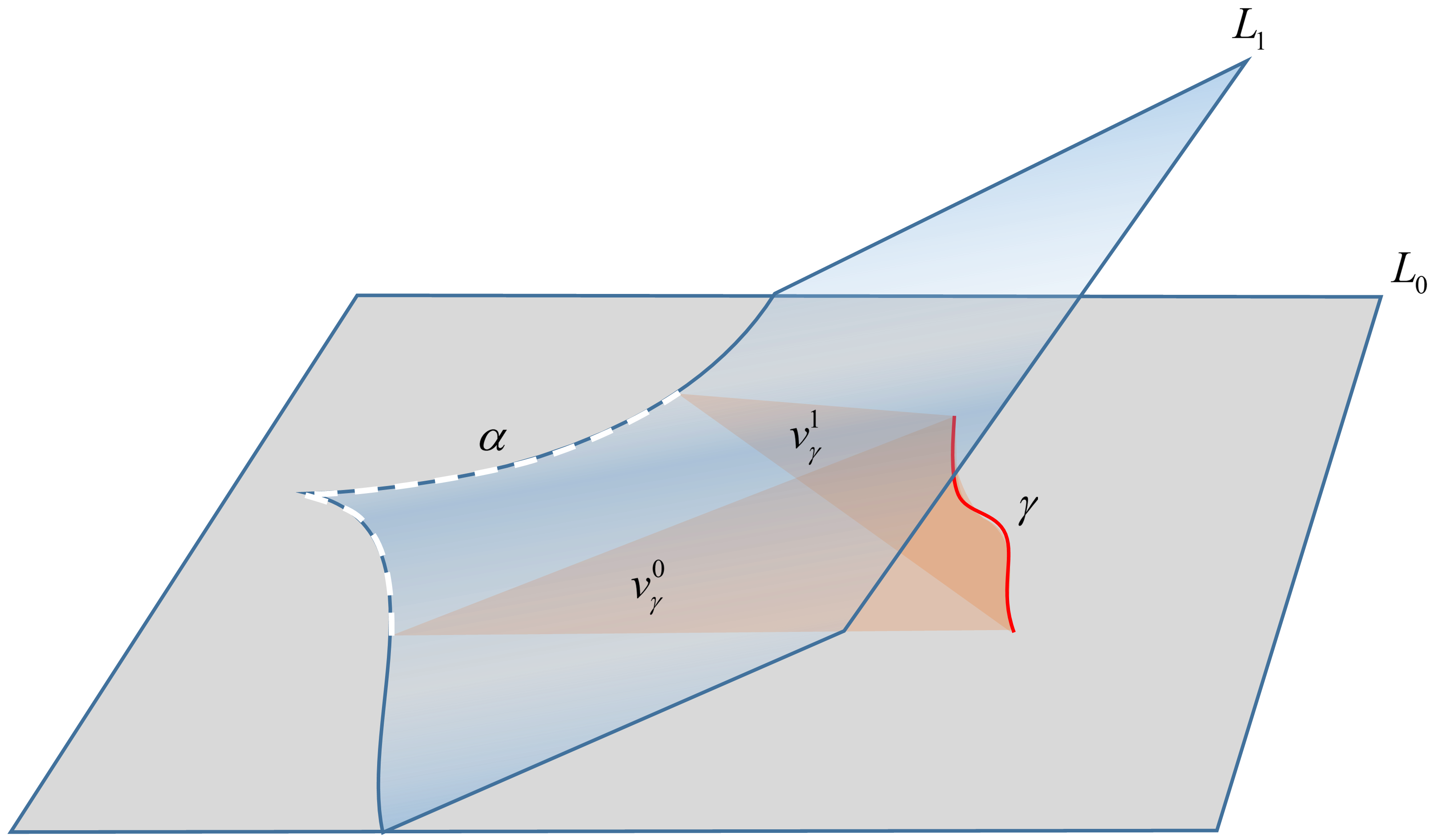}
\caption{A situation in which the path $\alpha$ must pass through a singularity of $L_0 \cap L_1.$ Nonetheless, its length is controlled by Whitney regularity.}
\label{fig:claimb}
\end{figure}
\begin{proof}
Inequality~\eqref{eq:e1}, Lemma~\ref{lm:gbp} and Corollary~\ref{cy:trip} give~\ref{it:vg}.

It remains to prove~\ref{it:inv}. Suppose $v_\gamma^j : H \to B_\epsilon(\gamma(0))$ for $j = 0,1,$ are two maps satisfying the conditions of~\ref{it:vg}. In particular,
\[
\dist_g\left(v_\gamma^0(0),v_\gamma^1(0)\right) < 2\epsilon.
\]
It follows from inequality~\eqref{eq:cc} that $v_\gamma^0(0)$ and $v_\gamma^1(0)$ belong to the same component of $L_0 \cap L_1.$ By Corollary~\ref{cy:cs} and inequality~\eqref{eq:wr}, choose a piecewise smooth $\alpha:[0,1] \to L_0\cap L_1 \subset M$ such that $\alpha(0) = v_\gamma^0(0)$ and $\alpha(1) = v_\gamma^1(0)$ and
\begin{equation}\label{eq:al}
\ell_g(\alpha)< \frac{\epsilon_1}{9}.
\end{equation}
See Figure~\ref{fig:claimb}. By inclusion~\eqref{eq:LM}, we have
\begin{gather*}
v_\gamma^j([0,1]) \subset B_{\epsilon}(\gamma(0)) \subset B_{2\epsilon}(\alpha(j)) \subset B_{4\epsilon}(\alpha(j),L_0),\\
v_\gamma^j([-1,0]) \subset B_\epsilon(\gamma(0)) \subset B_{2\epsilon}(\alpha(j)) \subset B_{4\epsilon}(\alpha(j),L_1).
\end{gather*}
By inequality~\eqref{eq:e1}, we may replace $4\epsilon$ by $2\epsilon_1/9.$ Moreover, by~\eqref{eq:al} we have
\[
B_{4\epsilon}(\alpha(1),L_i) \subset B_{\epsilon_1/3}(\alpha(0),L_i).
\]
For $i = 0,1,$ apply Lemma~\ref{lm:tri} with
\begin{gather*}
X = L_i, \quad \delta = \epsilon_1/3, \quad \chi = \alpha, \\
\beta_j(x) =
\begin{cases}
v_\gamma^j(1-x), & i = 0, \\
v_\gamma^j(x-1), & i = 1,
\end{cases}
\qquad x \in [0,1],
\end{gather*}
to obtain
\[
\sigma_i : H \to B_{\epsilon_1}(\alpha(0),L_i)
\]
such that for $x \in [0,1],$
\begin{gather*}
\sigma_0(x) = v_\gamma^0(1-x), \qquad \sigma_0(-x) = v_\gamma^1(1-x), \\
\sigma_1(x) = v_\gamma^0(x-1), \qquad \sigma_1(-x) = v_\gamma^1(x-1),
\end{gather*}
and
\[
\sigma_i(e^{\sqrt{-1}\pi t}) = \alpha(t).
\]
By the Poincar\'e lemma, $\omega|_{B_{\epsilon_1}(\alpha(0))}$ is exact. Since $\sigma_i,v_\gamma^j,$ all map into $B_{\epsilon_1}(\alpha(0)),$ Stokes' theorem implies
\[
\int_H (v_\gamma^1)^*\omega - \int_H (v_\gamma^0)^*\omega = \int_H \sigma_0^*\omega - \int_H \sigma_1^*\omega.
\]
On the other hand, since $\sigma_j$ maps into the Lagrangian $L_j,$ we have $\sigma_j^*\omega = 0.$
\end{proof}

We will also need the following proposition, which is based on Theorem~\ref{tm:prec}. Let
\[
\kappa: [a,b]\times [0,1] \to M, \qquad  \kappa([a,b]\times\{i\}) \subset L_i, \qquad i = 0,1.
\]
For $s \in [a,b],$ define $\kappa_s : [0,1] \to M$ by $\kappa_s(t) = \kappa(s,t).$
\begin{pr}\label{pr:lsa}
Let $\epsilon_0,\delta_0(\epsilon_0),$ be as in Theorem~\ref{tm:prec}.
Suppose that for all $s \in [a,b],$ either
\[
\ell_g(\kappa_s) < \delta_0(\epsilon_0)
\]
or there exists $p(s) \in L_0 \cap L_1$ such that
\begin{equation}\label{eq:nus}
\kappa_s([0,1]) \subset B_{\frac{\epsilon_0}{9}}(p(s)).
\end{equation}
Then
\[
\int_{[a,b]\times [0,1]} \kappa^*\omega = a(\kappa_1) - a(\kappa_0).
\]
\end{pr}
\begin{proof}
Equation~\eqref{eq:d0} and Lemma~\ref{lm:gbp} imply that for all $s \in [0,1]$ there exists $p(s)$ such that inclusion~\eqref{eq:nus} holds. A compactness argument yields $a \leq s_0 < s_1 < \cdots < s_N \leq b$ and $\vartheta_i > 0$ for $i = 0,\ldots,N,$ such that
\begin{gather}
[a,b] \subset \bigcup_{i = 1}^N (s_i - \vartheta_i, s_i + \vartheta_i), \notag\\
\label{eq:kinc}
\kappa_s([0,1]) \subset B_{\epsilon_0/9}(p(s_i)), \qquad s \in (s_i - \vartheta_i, s_i +\vartheta_i).
\end{gather}
Without loss of generality, we may assume $s_i-\vartheta_i < s_{i-1} + \vartheta_{i-1}$ and choose $\varsigma_i \in (s_i - \vartheta_i, s_{i-1}+\vartheta_{i-1})$ for $i = 1,\ldots,N,$ such that $\varsigma_1 < \cdots < \varsigma_N.$ Write also $\varsigma_0 = 0$ and $\varsigma_{N+1} = 1.$  We show that
\begin{equation}\label{eq:ki}
\int_{[\varsigma_i,\varsigma_{i+1}]\times[0,1]} \kappa^*\omega = a(\kappa_{\varsigma_{i+1}}) - a(\kappa_{\varsigma_i}).
\end{equation}
for $i = 0,\ldots,N.$

Indeed, for $k = 0,1,$ abbreviate $\kappa_i^k = \kappa_{\varsigma_{i+k}}.$ By inclusion~\eqref{eq:kinc}, apply Corollary~\ref{cy:trip} with $\gamma = \kappa_i^k,$ and $p = p(s_i)$ to obtain
\begin{equation*}
v_i^k := v_{\kappa_i^k}: H \to B_{\epsilon_0}\left(\kappa_i^k(0)\right),
\end{equation*}
with $v_i^k(0) = p(s_i).$ See Figure~\ref{fig:prop}.

\begin{figure}[ht]
\centering
\includegraphics[width=12cm]{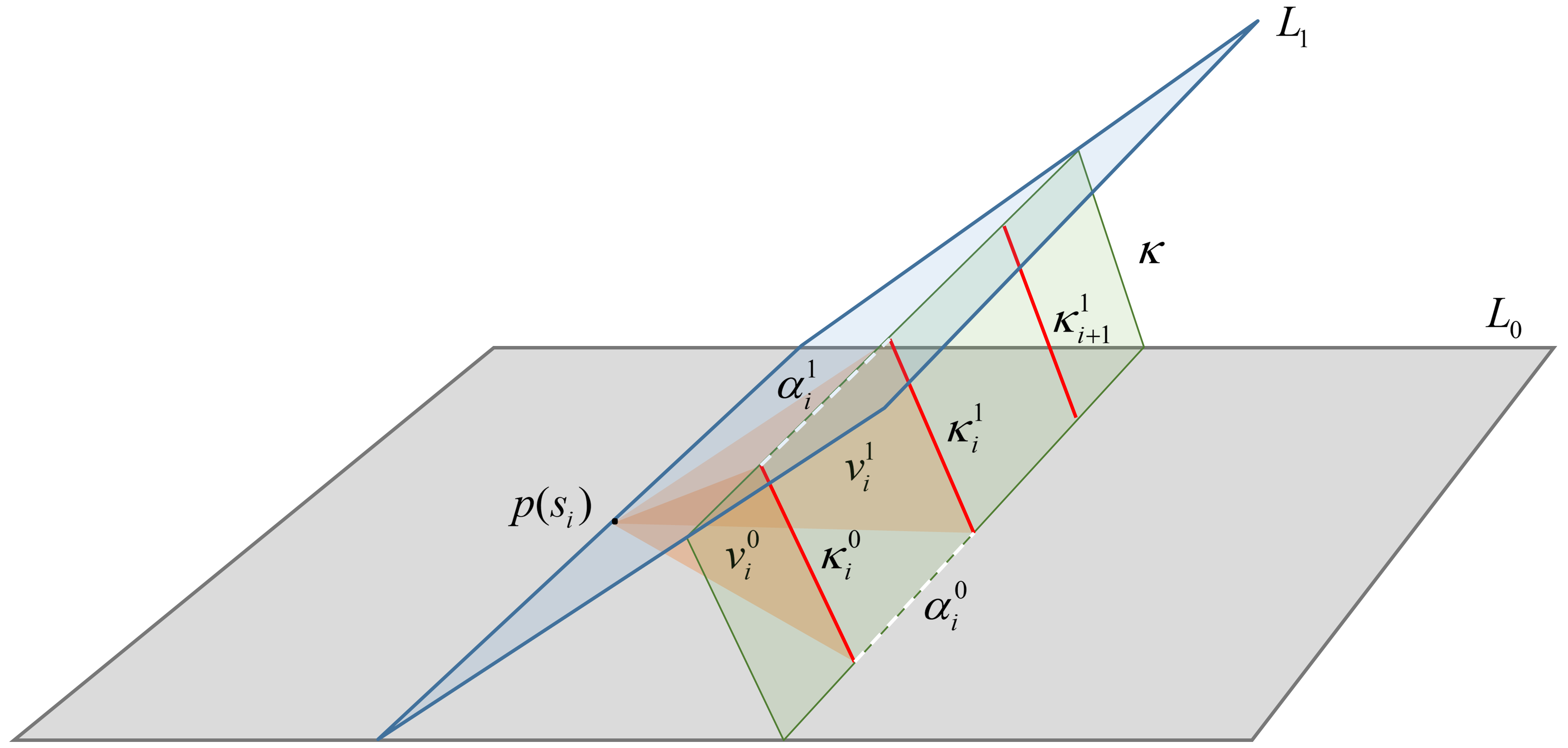}
\caption{}
\label{fig:prop}
\end{figure}

Furthermore, for $j = 0,1,$ let
\[
\alpha_i^j : [0,1] \to L_j
\]
be given by
\[
\alpha_i^j(s) = \kappa((1-s)\varsigma_i + s \varsigma_{i+1},j).
\]
By inclusions~\eqref{eq:kinc} and~\eqref{eq:LM}, we have
\begin{gather*}
v_i^k([0,1]) \subset B_{\epsilon_0}(\kappa_i^k(0)) \subset B_{2\epsilon_0}(\alpha_i^0(0)) \subset B_{4\epsilon_0}(\alpha_i^0(0),L_0),\\
v_i^k([-1,0]) \subset B_{\epsilon_0}(\kappa_i^k(0)) \subset B_{2\epsilon_0}(\alpha_i^1(0)) \subset B_{4\epsilon_0}(\alpha_i^1(0),L_1).
\end{gather*}
By inequality~\eqref{eq:e1}, we have $4\epsilon_0 \leq 2\epsilon_1/9.$
For $j = 0,1,$ apply Lemma~\ref{lm:tri} with
\begin{gather*}
X = L_j, \quad \delta = 2\epsilon_1/9, \quad \chi = \alpha^j_i, \\
\beta_k(x) =
\begin{cases}
v_i^k(x), & i = 0,\\
v_i^k(-x), & i = 1,
\end{cases}
\qquad x \in [0,1],
\end{gather*}
to obtain
\[
\sigma_i^j : H \to B_{2\epsilon_1/3}(\alpha^j_i(0),L_j)
\]
such that for $x \in [0,1],$
\begin{gather*}
\sigma_i^0(x) = v_i^0(x), \qquad \sigma_i^0(-x) = v_i^1(x), \\
\sigma_i^1(x) = v_i^0(-x), \qquad \sigma_i^1(-x) = v_i^1(-x),
\end{gather*}
and
\[
\sigma_i^j(e^{\sqrt{-1}\pi t}) = \alpha_i^j(t).
\]
Inclusion~\eqref{eq:kinc} implies that $\kappa([\varsigma_i,\varsigma_{i+1}]\times[0,1])$ and, in particular, $\kappa_i^1(0)$ and $\kappa_i^0(1) = \alpha_i^1(0),$ are all contained in the $\frac{2\epsilon_0}{9}$ ball centered at $\kappa_i^0(0) = \alpha_i^0(0).$ It follows from inequality~\eqref{eq:e1} that $\sigma_i^j,v_i^k,$ and $\kappa|_{[\varsigma_i,\varsigma_{i+1}]\times[0,1]},$ all map into $B_{\epsilon_1}(\alpha_i^0(0)).$
By the Poincar\'e lemma, $\omega|_{B_{\epsilon_1}(\alpha_i^0(0))}$ is exact. So, Stokes' theorem implies
\[
\int_{[\varsigma_i,\varsigma_{i+1}]\times[0,1]} \kappa^*\omega =\int_H (v_i^1)^*\omega - \int_H (v_i^0)^*\omega + \int_H (\sigma_i^0)^*\omega - \int_H (\sigma_i^1)^*\omega.
\]
On the other hand, since $\sigma_j$ maps into the Lagrangian $L_j,$ we have $\sigma_j^*\omega = 0.$ Moreover, by definition,
\[
\int_H(v_i^k)^*\omega = a(\kappa_{\varsigma_{i+k}}).
\]
Equation~\eqref{eq:ki} follows. Finally, by~\eqref{eq:ki} we have
\[
\int_{[a,b]\times [0,1]}\kappa^*\omega = \sum_{i=0}^N \int_{[\varsigma_i,\varsigma_{i+1}]\times[0,1]} \kappa^*\omega = a(\kappa_{\varsigma_{N+1}}) - a(\kappa_{\varsigma_0}).
\]
\end{proof}

\section{{\L}ojasiewicz isoperimetric inequality}
\label{sec:lii}
We continue with the assumptions and notation at the beginning of Section~\ref{sec:lsa}.
The main result of this section is the following isoperimetric inequality for the local symplectic action.
\begin{tm}\label{tm:lii}
There exist constants $\delta_1 > 0$ and $\beta > 1$ such that for all $\gamma: [0,1] \to M$ with $\gamma(0) \in L_0, \gamma(1) \in L_1$ and $\ell_g(\gamma) < \delta_1$, the local symplectic action $a(\gamma)$ is well defined and
\begin{equation}\label{eq:isoper}
|a(\gamma)| \leq \ell_g(\gamma)^\beta.
\end{equation}
\end{tm}
The proof of Theorem~\ref{tm:lii} is given below. A key ingredient is \L ojasiewicz's gradient inequality~\cite{_Loj65_}.
\begin{tm}[\L ojasiewicz]\label{tm:gi}
Let $f$ be an analytic function on a neighborhood of $0 \in \R^n.$ Then there exists $\theta \in (0,1)$ and a possibly smaller neighborhood $U$ of $0$ such that
\begin{equation}\label{eq:gi}
|\nabla f(y)| \geq |f(y)-f(0)|^\theta, \qquad  y \in U.
\end{equation}
\end{tm}
The proof of Theorem~\ref{tm:gi} uses the theory of semi-analytic sets. A nice exposition is given in~\cite[Prop. 6.8]{_BM88_}.

Let $z_i = x_i + \i y_i$ denote the standard coordinates of $\C^n.$ Let
\[
\omega_0 = \sum_{i = 1}^n dx_i \wedge dy_i
\]
denote the standard symplectic structure on $\C^n$ and let $g_0$ denote the Euclidean metric. Given an open subset $U \subset \R^n$ and $f: U \to \R,$ write
\[
\graph(\nabla f) = \{x + \i y \in \C^n| x \in U, \; y = \nabla f(x)\}.
\]
So, $\graph(\nabla f) \subset \C^n$ is a Lagrangian submanifold.
\begin{lm}\label{lm:llii}
Let $W \subset \C^n$ be an open set containing the origin, and let $f$ be a real analytic function defined on a neighborhood of the origin in $\R^n$ with $\nabla f(0) = 0.$ Let
\[
Q_0 = \R^n \cap W, \quad Q_1 = \graph(\nabla f).
\]
There exist $\beta > 1$ and $\epsilon > 0$ with the following significance.
Suppose $\gamma : [0,1] \to B_\epsilon(0) \subset \C^n$ with $\gamma(i) \in Q_i$ for $i = 0,1,$ and $\ell_{g_0}(\gamma) < 1.$ Let
$v_\gamma : H\to B_\epsilon(0)$ with
\begin{gather*}
v_\gamma([-1,0]) \subset Q_1, \qquad v_\gamma([0,1]) \subset Q_0,\\
v_\gamma(e^{\sqrt{-1}\pi t}) = \gamma(t).
\end{gather*}
Then
\[
\left | \int_H v_\gamma^*\omega_0 \right| \leq 2\ell_{g_0}(\gamma)^\beta.
\]
\end{lm}
\begin{proof}
Apply Theorem~\ref{tm:gi} to $f$ to obtain $\epsilon > 0$ such that the gradient inequality~\eqref{eq:gi} holds for all $y \in B_\epsilon(0) \subset \R^n.$ Take $\beta = 1/\theta.$
Write
\[
\lambda_0 = -\sum_i y_i dx_i,
\]
so $d\lambda_0 = \omega_0.$ Let $c : [0,1] \to Q_1$ be given by $c(x) = v_\gamma(-x).$ So, $c(0) \in Q_0 \cap Q_1$ and $c(1) = \gamma(1).$ Since $\lambda_0|_{Q_0} = 0,$ Stokes' theorem gives
\begin{equation}\label{eq:a1}
\int_H v_\gamma^*\omega_0 = -\int_0^1 c^*\lambda_0 + \int_0^1 \gamma^*\lambda_0.
\end{equation}
Write $\gamma(t) = k(t) + \i h(t)$ with $k,h : [0,1] \to \R^n.$ Since $\gamma(0) \in Q_0$, we have $h(0)=0$ and thus
\begin{equation}\label{eq:betabd}
|h(t)| = |h(t) - h(0)| \leq \ell_{g_0}(\gamma).
\end{equation}
Let $\pi : \C^n \to \R^n$ denote the projection given by $\pi(x+ \i y) = x.$ So,
\[
\lambda_0|_{Q_1} = -\pi^*df|_{Q_1},
\]
and consequently
\[
-\int_{0}^1 c^*\lambda_0 = f(\pi(c(1))) - f(\pi(c(0))).
\]
Since $c(0) \in Q_0 \cap Q_1,$ we have $\pi(c(0)) = c(0)$ and $\nabla f(c(0)) = 0$. So, the gradient inequality~\eqref{eq:gi} implies $f(c(0)) = f(0).$ Furthermore, the gradient inequality gives
\begin{equation}\label{eq:lambda2}
\left |-\int_{0}^1 c^*\lambda_0\right| = |f(\pi(c(1))) - f(\pi(c(0)))| \leq |\nabla f(\pi(c(1))|^\beta.
\end{equation}
Since $c(1) = \gamma(1),$ we have $\pi(c(1)) = \pi(\gamma(1)) = k(1)$. Moreover, $\gamma(1) \in Q_1$ implies $h(1) = \nabla f(k(1))$. So, using inequality~\eqref{eq:betabd}, we obtain
\begin{equation}\label{eq:lambda3}
|\nabla f(\pi(c(1)))| = |h(1)|  \leq \ell_{g_0}(\gamma).
\end{equation}
Combining inequalities~\eqref{eq:lambda2} and~\eqref{eq:lambda3} we conclude
\begin{equation}\label{eq:lambda}
\left |-\int_{0}^1 c^*\lambda_0\right|\leq \ell_{g_0}(\gamma)^\beta.
\end{equation}
On the other hand, to calculate $\int_0^1\gamma^*\lambda_0$ we may assume without loss of generality that $\gamma$ is parameterized by arc-length, so $|\dot\gamma(t)| = \ell_{g_0}(\gamma).$ Thus, inequality~\eqref{eq:betabd} gives
\begin{equation}\label{eq:gamma}
\left|\int_0^1\gamma^*\lambda_0\right| = \left|-\int_{0}^1 h(t) \cdot \dot k(t) dt \right| \leq \ell_{g_0}(\gamma)^2.
\end{equation}
Combining~\eqref{eq:a1},~\eqref{eq:lambda} and~\eqref{eq:gamma}, we obtain the desired bound.
\end{proof}
\begin{proof}[Proof of Theorem~\ref{tm:lii}]
Equip $\C^n$ with the standard symplectic structure $\omega_0$ and the Euclidean metric $g_0.$
By Darboux' theorem, for each point $p \in L_0 \cap L_1,$ choose an open set $\widetilde V_p \subset M$ with $p \in \widetilde V_p,$ an open set $\widetilde W_p \subset \C^n$ with $0 \in \widetilde W_p$ and a real analytic Lipschitz symplectomorphism $\phi_p: \widetilde V_p \to \widetilde W_p$ with $\phi_p(p) = 0.$ After possibly modifying $\phi_p$ and shrinking $\widetilde V_p,\widetilde W_p,$ we may assume that $\phi_p(L_0) = \R^n \cap \widetilde W_p$ and
\[
\phi_p(L_1) = \graph(\nabla f_p),
\]
where $f_p$ is analytic function defined on a neighborhood of $0$ in $\R^n.$
Let $\beta_p>1$ and $\epsilon_p>0$ be the constants obtained from Lemma~\ref{lm:llii} with
\[
W = \widetilde W_p, \qquad  f = f_p.
\]
Write $W_p = B_{\epsilon_p}(0)\cap \widetilde W_p$ and $V_p = \phi_p^{-1}(W_p).$

By compactness of $L_0 \cap L_1,$ choose a finite collection of points
\[
p_1,\ldots,p_N \in L_0\cap L_1
\]
such that $L_0 \cap L_1 \subset \cup_{i = 1}^N V_{p_i}.$  Let $U_r$ denote the metric $r$-neighborhood of $L_0 \cap L_1$ with respect to $g.$ Choose $r_0 > 0$ small enough that
\[
\overline U_{r_0} \subset \cup_{i = 1}^N V_{p_i}.
\]
Let $\epsilon_2$ be a Lebesgue number for the cover $\{V_{p_i}\cap \overline U_{r_0}\}_{i = 1}^N$ of $\overline U_{r_0}.$ Let $\epsilon_0$ be as in Theorem~\ref{tm:vg} and let
\[
\epsilon = \min\left\{\epsilon_0,\epsilon_2,\frac{r_0}{2}\right\}.
\]
Let $\delta_1 = \min\{\delta_0(\epsilon),1\},$ and choose $\beta$ such that
\[
1 < \beta < \min\{\beta_{p_i}\}_{i = 1}^N.
\]

For any path $\gamma : [0,1] \to M$ with $\gamma(i) \in L_i,$ for $i = 0,1,$ and $\ell_g(\gamma) < \delta_1$, there exists $v_\gamma$ as in Theorem~\ref{tm:vg}\ref{it:vg}. Since $v_\gamma(0) \in L_0 \cap L_1$ and $\epsilon \leq \frac{r_0}{2},$ it follows that $v_\gamma(H) \subset U_{r_0}.$ Since $\epsilon \leq \epsilon_2,$ it follows that $v_\gamma(H) \subset V_{p_i}$ for some $i.$ Therefore,
\[
|a(\gamma)| \leq 2 \left(\ell_{g_0}(\phi_{p_i}\circ \gamma)\right)^{\beta_{p_i}} \leq C \ell_g(\gamma)^{\beta},
\]
where $C$ depends only on the Lipschitz constants of $\phi_{p_i}.$ Possibly making $\delta_1$ smaller, we eliminate the constant $C$ and obtain the desired inequality.
\end{proof}

\section{Continuous extension}
\label{sec:limit}


We continue with the assumptions and notation at the beginning of Section~\ref{sec:lsa}.
The main result of this section is the following theorem. Its proof and the resulting proof of Theorem~\ref{tm:conti} are given at the end of the section.

\begin{tm}\label{tm:limit}
Let $u: H \setminus \{0\} \to M$ be a $\Theta$-holomorphic map with $u([-1,0)) \subset L_1$ and $u((0,1]) \subset L_0.$ If $E_{g}(u) < \infty$, then $u$ extends to a continuous map $\bar u: H \to M.$ Moreover, there exist constants $c,\alpha,\epsilon > 0$ such that
\begin{equation}\label{eq:Hd}
|du(z)| \leq \frac{c}{|z|(\log{|z|})^\alpha}, \qquad z \in H, \quad |z| < \epsilon.
\end{equation}
\end{tm}
The proof of Theorem~\ref{tm:limit} relies on the following discussion and will be given at the end of the section. It combines ideas in the proof of~\cite[Theorem 4.1.2]{_MS_} with results from Sections~\ref{sec:lsa} and~\ref{sec:lii}.

The almost complex structure $\Theta$ is called $\omega$-\textbf{compatible} if for all tangent vectors $\xi,\eta \in T_pM$ we have $\omega(\Theta\xi,\Theta\eta) = \omega(\xi,\eta).$ The following lemma is of fundamental importance in Gromov's theory of $J$-holomorphic curves~\cite{_Gromov:curves_}. See~\cite[Lemma 2.2.1]{_MS_}.
\begin{lm}\label{lm:ei}
Let $\Sigma$ be a Riemann surface. Every $\Theta$-holomorphic map $u: \Sigma \to M$ satisfies
\[
E_g(u) = \int_\Sigma u^*\omega.
\]
If $\Theta$ is $\omega$-compatible, then for any smooth map $u: \Sigma\to M,$ we have
\[
E_g(u) := \int_\Sigma |du|^2 \dvol = \int_\Sigma |\bar\partial_\Theta u|^2 \dvol + \int u^*\omega.
\]
In fact, this follows from a pointwise equality of the integrands.
\end{lm}

Let
\[
D_r = \{z \in \C| |z| \leq r\}, \qquad H_r = \{z \in D_r | \Im z \geq 0\},
\]
denote the disk of radius $r$ and the upper half-disk of radius $r$ respectively. When it does not cause confusion, we abbreviate $|\cdot | = |\cdot|_g.$ When not indicated otherwise, integrals over subsets of $\C$ are with respect to the volume form of the Euclidean metric on $\C.$
We will use the following mean value inequality from~\cite[Lemma~4.3.1]{_MS_}.
\begin{lm}\label{lm:mvi}
Let $L \subset M$ be a totally real submanifold with respect to~$\Theta.$ There are constants $c_0,\delta_2>0$ such that the following holds.
If $r> 0$ and $u: (H_{2r}, H_{2r}\cap\R) \to (M,L)$ is $\Theta$-holomorphic, then
\[
\int_{H_{2r}} |du|^2 < \delta_2 \qquad \Rightarrow \qquad \sup_{z \in H_r}|du(z)|^2 \leq \frac{c_0}{r^2}\int_{H_{2r}} |du|^2.
\]
\end{lm}

For $\tau \in \R,$ abbreviate
\begin{gather*}
S_\tau = \{x + \sqrt{-1}y|x \in (-\infty,\tau], \; y \in [0,1]\} \subset \C, \\
\partial_i S_\tau = \{x + \sqrt{-1}y \in S_\tau| x = i\}, \qquad i = 0,1.
\end{gather*}
In the following lemmas, we consider a $\Theta$-holomorphic map
\[
v : S_0 \to M, \qquad u(\partial_i S_0) \subset L_i, \quad i = 0,1,
\]
with
\[
E_g(v) = \frac{1}{2} \int_{S_0} |dv|^2 < \infty.
\]
We abbreviate $v(x,y) = v(x + \sqrt{-1}y).$ For $x \in (-\infty,0],$ let $v_x : [0,1] \to M$ denote the path given by
\[
v_x(y) = v(x,y).
\]
Define $\e: (-\infty,0] \to \R_{\geq 0}$ by
\[
\e(\tau) = \frac{1}{2}\int_{S_\tau} |dv|^2.
\]
The proof of the following lemma uses ideas from~\cite[Lemma 4.5.1]{_MS_}
\begin{lm}\label{lm:combo}
Let $\beta$ be as in Theorem~\ref{tm:lii} and let $\delta_2$ be as in Lemma~\ref{lm:mvi}. There exists $\tau_0 \in (-\infty, 0]$ such that
\[
\tau \in (-\infty, \tau_0]\qquad  \Rightarrow \qquad \e(\tau) \leq \ell_g(v_\tau)^\beta, \quad \e(\tau) \leq \delta_2.
\]
\end{lm}
\begin{proof}
Let $\delta_0(\epsilon_0)$ be as in Proposition~\ref{pr:lsa}, let $\delta_1$ be as in Theorem~\ref{tm:lii} and let $\delta_2,c_0,$ be as in Lemma~\ref{lm:mvi}. Abbreviate
\[
\delta_3 = \min \{\delta_0(\epsilon_0),\delta_1\},
\]
and let $\delta$ satisfy
\[
0 < \delta \leq \delta_3.
\]
Observe that $\e$ is a continuous function with
\[
\lim_{\tau \to -\infty} \e(\tau) = 0.
\]
So, choose $\tau_\delta$ such that for $\tau \leq \tau_\delta+1$ we have
\[
\e(\tau) \leq \min\left\{\delta_2,\frac{\delta^2}{2c_0}\right\}
\]
Lemma~\ref{lm:mvi} with $r = \frac{1}{2}$ gives
\[
|dv(z)| \leq \sqrt{2}\delta, \qquad z \in S_{\tau_\delta}.
\]
Since $v$ is $\Theta$ holomorphic, it follows that
\[
|\dot v_\tau(t)| \leq \delta, \qquad \tau \leq \tau_\delta, \; t \in [0,1],
\]
and consequently,
\[
\ell_g(v_\tau) = \int_0^1 |\dot v_\tau(t)|dt \leq \delta, \qquad \tau \leq \tau_\delta.
\]
Let $\tau_0 = \tau_{\delta_3}$ and let $\tau < \tau_0.$ Then, Lemma~\ref{lm:ei},
Proposition~\ref{pr:lsa} with $\kappa = v|_{[\tau_\delta,\tau]}$ and Theorem~\ref{tm:lii}, give
\[
\frac{1}{2}\int_{[\tau_\delta,\tau]} |dv|^2 = \int_{[\tau_\delta,\tau]\times [0,1]} v^*\omega = a(v_\tau) - a(v_{\tau_\delta}) \leq \ell_g(v_\tau)^\beta + \delta^\beta.
\]
Since $\delta$ can be taken arbitrarily small, and we may assume without loss of generality that $\tau_\delta \to -\infty$ as $\delta \to 0,$ the lemma follows.
\end{proof}

For a smooth path $\gamma : [0,1] \to M,$ the energy is given by
\[
E(\gamma) = \int_0^1 |\dot \gamma(t)|^2 dt.
\]
\begin{lm}\label{lm:edot}
We have
\[
\dot \e(\tau) = E(v_\tau).
\]
\end{lm}
\begin{proof}
Using the fact that $v$ is $\Theta$ holomorphic and Fubini's theorem, we obtain
\begin{multline*}
\e(\tau) = \frac{1}{2}\bigintsss_{S_\tau} \left(\left| \frac{\partial v}{\partial x}\right|^2 + \left|\frac{\partial v}{\partial y}\right|^2\right) = \\
= \bigintssss_{S_\tau} \left|\frac{\partial v}{\partial y}\right|^2  =  \bigintssss_{-\infty}^\tau\left(\bigintssss_0^1 \left|\frac{\partial v}{\partial y}\left(x,y\right)\right|^2 dy\right)dx = \int_{-\infty}^\tau E(v_x)dx.
\end{multline*}
\end{proof}
The proof of the following lemma uses ideas from~\cite[Theorem~ 4.1.2]{_MS_}.
\begin{lm}\label{lm:decay}
With $\tau_0$ as in Lemma~\ref{lm:combo}, there exist constants  $c_1,\alpha_1 > 0,$ such that
\[
\e(\tau) \leq  \frac{c_1}{(\tau_0-\tau)^{\alpha_1}}, \qquad \tau \in (-\infty,\tau_0].
\]
\end{lm}
\begin{proof}
Write $\theta = \frac{1}{\beta}$ and let $\tau < \tau_0.$ Without loss of generality, we may assume that $\theta > \frac{1}{2}.$ By Lemma~\ref{lm:combo}, Holder's inequality and Lemma~\ref{lm:edot}, we obtain
\[
\e(\tau)^{2\theta} \leq \ell_g(v_\tau)^2 \leq E(v_\tau) = \dot \e(\tau).
\]
Integrating this differential inequality from $\tau$ to $\tau_0$ gives
\[
(1-2\theta)(\tau_0-\tau) \geq \e(\tau_0)^{1-2\theta} - \e(\tau)^{1-2\theta}.
\]
So,
\[
\e(\tau)^{1-2\theta} \geq (2\theta - 1) (\tau_0 - \tau)
\]
and consequently,
\[
\e(\tau) \leq c_1(\tau_0 - \tau)^{\frac{-1}{(2\theta -1)}}
\]
with $c_1= (2\theta - 1)^{\frac{-1}{(2\theta -1)}}.$
\end{proof}

\begin{cy}\label{cy:decay}
With $\tau_0$ as in Lemma~\ref{lm:combo} and $\alpha_1$ as in Lemma~\ref{lm:decay}, there exists a constant $c_2 > 0$ such that
\[
\ell_g(v_\tau) \leq \frac{c_2}{(\tau_0 - \tau-1)^{\frac{\alpha_1}{2}}}, \qquad \tau \in (-\infty,\tau_0 - 1].
\]
Moreover,
\[
|dv(\tau,t)| \leq \frac{c_2}{(\tau_0 - \tau-1)^{\frac{\alpha_1}{2}}}, \qquad \tau \in (-\infty,\tau_0 - 1], \quad t \in [0,1].
\]
\end{cy}
\begin{proof}
Let $t \in (\infty,\tau_0-1].$ Lemma~\ref{lm:mvi} with $r = \frac{1}{2}$ and Lemma~\ref{lm:decay} imply that
\[
|dv(\tau,t)| \leq \sqrt{4c_0\e(\tau+1)} \leq \frac{2\sqrt{c_0c_1}}{(\tau_0-\tau-1)^{\frac{\alpha_1}{2}}}.
\]
Since $|\dot v_\tau(t)| = \frac{1}{\sqrt{2}}|dv(\tau,t)|,$ we obtain
\[
\ell_g(v_\tau) = \int_0^1 |\dot v_\tau(t)| dt \leq \frac{\sqrt{2 c_0c_1}}{(\tau_0-\tau-1)^{\frac{\alpha_1}{2}}}.
\]
\end{proof}
The following lemma uses an idea of~\cite{_Loj84_}. See also~\cite{_KM_}.
\begin{lm}\label{lm:limit}
With $\tau_0$ as in Lemma~\ref{lm:combo}, there exists $p \in L_0 \cap L_1$ and constants $c_3,\alpha_2 > 0,$ such that
\[
d_g\!\left(v\left(\tau,t\right),p\right) \leq \frac{c_3}{(\tau_0 - \tau-1)^{\alpha_2}}, \qquad \tau \in (-\infty,\tau_0 - 1].
\]
\end{lm}
\begin{proof}
Set $\theta = \frac{1}{\beta}$ and let $\tau \in (-\infty,\tau_0-1].$ By Lemma~\ref{lm:combo} we have
\[
\e(\tau)^{-\theta} \geq \frac{1}{\ell_g(v_{\tau})}.
\]
Thus, by Lemma~\ref{lm:edot} and Holder's inequality, we obtain
\begin{equation}\label{eq:cntl}
\frac{d}{d\tau} \e(\tau)^{1-\theta} = (1-\theta)\e(\tau)^{-\theta}E(v_\tau) \geq \frac{E(v_\tau)}{\ell_g(v_\tau)}\geq \ell_g(v_\tau).
\end{equation}
Let $\tau' \in (-\infty,\tau].$ Inequality~\eqref{eq:cntl}, the fact that $v$ is $\Theta$-holomorphic, and Tonelli's theorem give
\begin{multline*}
\e(\tau)^{1-\theta} - \e(\tau')^{1-\theta} \geq \int_{\tau'}^{\tau}\ell_g(v_x)dx = \int_{\tau'}^{\tau}\left(\int_0^1 \left|\frac{\partial v}{\partial y}(x,y)\right|dy\right)dx = \\
=\int_{\tau'}^{\tau}\left(\int_0^1 \left|\frac{\partial v}{\partial x}(x,y)\right|dy\right)dx = \int_0^1 \left(\int_{\tau'}^{\tau}\left|\frac{\partial v}{\partial x}(x,y)\right|dx\right)dy.
\end{multline*}
It follows that there exists $y \in [0,1]$ such that
\[
d(v(\tau,y),v(\tau',y)) \leq \int_{\tau'}^{\tau} \left|\frac{\partial v}{\partial x}(\tau,y)\right|d\tau \leq \e(\tau)^{1-\theta} - \e(\tau')^{1-\theta}.
\]
So, with $\alpha_2 = (1-\theta)\alpha_1,$ Lemma~\ref{lm:decay} implies
\[
d_g(v(\tau,y),v(\tau',y)) \leq c_1^{(1-\theta)}\left(\frac{1}{(\tau_0 - \tau)^{\alpha_2}} - \frac{1}{(\tau_0-\tau')^{\alpha_2}}\right).
\]
Finally, for arbitrary $t,t' \in [0,1],$ Corollary~\ref{cy:decay} gives
\begin{multline}\label{eq:Cauchy}
d_g(v(\tau,t),v(\tau',t')) \leq \ell_g(v_{\tau}) + d_g(v(\tau,y),v(\tau',y)) + \ell_g(v_{\tau'}) \leq \\
\leq(c_1^{(1-\theta)} + c_2)\left(\frac{1}{(\tau_0 - \tau-1)^{\alpha_2}} - \frac{1}{(\tau_0-\tau'-1)^{\alpha_2}}\right).
\end{multline}
Let $(\tau_i,t_i) \in S_0$ with $\lim_{i\to \infty} \tau_i = -\infty.$ Inequality~\eqref{eq:Cauchy} implies that $v(\tau_i,t_i)$ is a Cauchy sequence. Let
\[
p = \lim_{i \to \infty} v(\tau_i,t_i).
\]
Set $(\tau',t') = (\tau_i,t_i)$ in inequality~\eqref{eq:Cauchy} and take the limit as $i \to \infty$ to obtain the desired inequality.
\end{proof}

\begin{proof}[Proof of Theorem~\ref{tm:limit}]
Define $v: S_0 \to M$ by $v(z) = u(e^{\pi z}).$ With $p$ as in Lemma~\ref{lm:limit}, define $\bar u(0) = p.$ Estimate~\eqref{eq:Hd} follows from Corollary~\ref{cy:decay}.
\end{proof}

\begin{proof}[Proof of Theorem~\ref{tm:conti}]
For each point $p \in \overline \Sigma \setminus \Sigma,$ choose a map
\[
\phi: (H,H \cap \R) \to (\overline\Sigma,\partial \overline\Sigma)
\]
with $\phi(0) = p$ and $\phi$ a biholomorphism onto its image. Apply Theorem~\ref{tm:limit} to the map $u\circ \phi : H\setminus\{0\} \to M.$
\end{proof}

\section{Homotopy classes}
\label{sec:hc}
For the proof of Theorem~\ref{tm:main}, we need a couple more lemmas of a topological nature, which are the subject of the present section. Consider a nice Riemann surface $\Sigma$ as well as a symplectic manifold $(M,\omega)$ and a collection $\LL$ of Lagrangian submanifolds $L_i \subset M,\, i \in A,$ all as in Section~\ref{sec:intro}. Let $l$ be an $A$-labeling of $\Sigma.$ Assume, moreover, that $M,\omega$ and the Lagrangian submanifolds $L_i$ are real analytic. Let $\Map(\overline \Sigma,(M,l))$ denote the space of continuous maps $f: \overline \Sigma \to M$ such that $f(\partial\Sigma_i) \subset L_{l(i)},$ endowed with the compact open topology. Let $[\overline \Sigma,(M,l)]$ denote the set of homotopy classes of such maps, that is, the set of path connected components of~$\Map(\overline \Sigma,(M,l)).$ For $n \in \N,$ abbreviate $[n]:= \{1,\ldots,n\}.$

\begin{lm}\label{lm:count}
The set $[\overline \Sigma,(M,l)]$ is countable.
\end{lm}
\begin{proof}
By \L ojasiewicz's triangulation theorem~\cite{_Loj64_}, choose a countable simultaneous triangulation of $M$ and $L_{l(i)}$ for $i \in [b(\Sigma)].$ Choose also a finite simultaneous triangulation of $\overline \Sigma$ and $\overline{\partial \Sigma_i}$ for $i \in [b(\Sigma)]$. Let $f \in \Map(\overline \Sigma,(M,l)).$  By Theorem 3.4.8 and Corollary 3.4.4 in~\cite{_Sp_}, after possible iterated barycentric subdivision of the triangulation of $\overline \Sigma,$ there exists a simplicial approximation $g\in \Map(\overline \Sigma,(M,l))$ of $f.$ Examining the proof of Lemma 3.4.2 in~\cite{_Sp_}, we see that
\[
[g] = [f] \in [\overline \Sigma,(M,l)].
\]
Since the set of simplicial maps from an iterated barycentric subdivision of a fixed finite simplicial complex to a fixed countable simplicial complex is countable, the lemma follows.
\end{proof}

In the following, $g = g_\Theta$ for $\Theta$ an $\omega$-compatible complex structure, and $h$ is a Hermitian metric on $\overline \Sigma.$
\begin{df}\label{df:as}
We say a continuous map $f \in \Map(\overline \Sigma,(M,l))$ is \textbf{almost smooth} if $\hat f := f|_\Sigma$ is smooth, the energy $E_g(\hat f)$ is finite, and there exist constants $c,\alpha,\epsilon > 0$ such that
\begin{equation}\label{eq:fd}
|df|_{g,h}(z) < \frac{c}{d_h(p,z) (\log d_h(p,z))^\alpha}, \qquad z \in \Sigma, \quad d_h(z,p) < \epsilon.
\end{equation}
Let $[\overline\Sigma,(M,l)]^{as}\subset [\overline\Sigma,(M,l)]$ denote the set of homotopy classes that admit an almost smooth representative.
\end{df}

\begin{lm}\label{lm:as}
There is a well-defined map
\[
I_\omega : [\overline\Sigma,(M,l)]^{as} \to \R
\]
given by
\[
[f] \to \int_\Sigma \hat f^*\omega.
\]
for $f$ an almost smooth representative.
\end{lm}
\begin{proof}
By Lemma~\ref{lm:ei}, the integral $\int_\Sigma \hat f^* \omega$ is well defined and finite. We must show that if $f_j\in \Map(\overline \Sigma,(M,l))$ for $j = 0,1,$ are almost smooth and $[f_0] = [f_1] \in [\overline\Sigma,(M,l)],$ then
\begin{equation}\label{eq:hinv}
\int_\Sigma \hat f_0^*\omega = \int_\Sigma \hat f_1^*\omega.
\end{equation}
By assumption we have a continuous map
\[
F : [0,1] \times \overline \Sigma \to M
\]
such that
\begin{equation}\label{eq:Fc}
F|_{\{j\} \times \overline \Sigma} = f_j, \quad j = 0,1, \qquad F([0,1]\times \partial\Sigma_i) \subset L_{l(i)}, \quad i \in [b(\Sigma)].
\end{equation}
If $F$ were smooth, the lemma would follow immediately from Stokes' theorem.  However, since there is no assumption on how the Lagrangian submanifolds $L_i$ intersect each other, it appears unreasonable to expect to find a homotopy $F$ that is smooth on the whole domain $\overline \Sigma \times [0,1].$ Rather we approximate $F$ by a smooth map on $\Sigma\times [0,1]$ and use Proposition~\ref{pr:lsa} together with the continuity of $F$ to work around the possible wild behavior of the approximation near points in~$\overline \Sigma \setminus \Sigma.$

Indeed, let $\epsilon_0$ be as in Theorem~\ref{tm:prec}, let $\beta,\delta_1,$ be as in Theorem~\ref{tm:lii} and let $\varrho > 0$ be arbitrary. Using the techniques of~\cite[Chapter 2]{_Hi_}, one can show there exists a smooth map
\[
G : [0,1]\times \Sigma \to M
\]
such that
\begin{gather}
G|_{\{j\} \times \Sigma} = \hat f_j, \quad j = 0,1, \notag \\ G([0,1]\times \partial\Sigma_i) \subset L_{l(i)}, \quad i \in [b(\Sigma)], \label{eq:lbc}
\end{gather}
and
\begin{equation}\label{eq:FG}
d_{g_\Theta}(F(t,z),G(t,z)) < \frac{\epsilon_0}{18}, \qquad t \in [0,1],\quad z \in \Sigma.
\end{equation}
For each $p \in \overline \Sigma \setminus \Sigma,$ choose a neighborhood $U_p$ biholomorphic to $H$ such that $U_p \cap U_q = \emptyset$ for $p \neq q$, so
\[
\widetilde \Sigma : = \overline \Sigma \setminus \left(\cup_p U_p\right)
\]
is a Riemann surface with corners. Let $I_p \subset \Sigma$ denote the topological boundary of $U_p,$ that is, the points in $U_p$ corresponding to the semi-circle in the boundary of $H.$ Thus, $I_p$ is diffeomorphic to a closed interval.
It follows from condition~\eqref{eq:fd} that by choosing $U_p$ sufficiently small, we may assume that
\begin{equation}\label{eq:lfj}
\ell_{g_\Theta}(f_j|_{I_p}) < \delta_1, \qquad \ell_{g_\Theta}(f_j|_{I_p})^\beta < \varrho.
\end{equation}
Since $E_g(\hat f_j) < \infty,$ we may choose $U_p$ small enough that
\begin{equation}\label{eq:efjUp}
\left | \int_{U_p\setminus \{p\}} \hat f^*_j \omega \right| < \varrho, \qquad p \in \overline \Sigma \setminus \Sigma, \quad j = 0,1.
\end{equation}
Furthermore, by continuity of $F,$ choosing $U_p$ sufficiently small, we may assume that
\begin{equation}\label{eq:Fzp}
d_{g_\Theta}(F(t,z), F(t,p))< \frac{\epsilon_0}{18}, \qquad t \in [0,1],\quad z \in I_p.
\end{equation}
Stokes' theorem and the Lagrangian boundary conditions~\eqref{eq:lbc} give
\begin{equation}\label{eq:StoG}
0 = \int_{[0,1]\times \widetilde\Sigma} G^*d\omega = \int_{\widetilde \Sigma}\hat f_1^*\omega - \int_{\widetilde\Sigma}\hat f_0^*\omega + \sum_{p \in \overline \Sigma\setminus \Sigma} \int_{[0,1]\times I_p} F^*\omega.
\end{equation}
By inequalities~\eqref{eq:FG} and~\eqref{eq:Fzp}, Proposition~\ref{pr:lsa} with $\kappa = F|_{[0,1]\times I_p}$ gives
\begin{equation*}
\int_{[0,1]\times I_p} F^*\omega  = a(f_1|_{I_p}) - a(f_0|_{I_p}).
\end{equation*}
It then follows from Theorem~\ref{tm:lii} and the inequalities~\eqref{eq:lfj} that
\begin{equation}\label{eq:Flii}
\left|\int_{[0,1]\times I_p} F^*\omega \right|< 2\varrho.
\end{equation}
Equation~\eqref{eq:StoG} and inequality~\eqref{eq:Flii} imply that
\begin{equation}\label{eq:St}
\left|\int_{\widetilde \Sigma}\hat f_1^*\omega - \int_{\widetilde\Sigma}\hat f_0^*\omega\right| < 2\left|\overline \Sigma \setminus \Sigma\right|\varrho.
\end{equation}
On the other hand, by inequality~\eqref{eq:efjUp}, for $j = 0,1,$ we have
\begin{equation}\label{eq:rem}
\left|\int_{\Sigma} \hat f_j^*\omega  - \int_{\widetilde\Sigma} \hat f_j^*\omega\right| = \left|\sum_{p \in \overline \Sigma \setminus \Sigma} \int_{U_p\setminus\{p\}} \hat f_j^*\omega\right| < \left|\overline \Sigma \setminus \Sigma\right|\varrho.
\end{equation}
Combining inequalities~\eqref{eq:St} and~\eqref{eq:rem}, we obtain
\[
\left|\int_{\Sigma}\hat f_1^*\omega - \int_{\Sigma}\hat f_0^*\omega\right| < 4\left|\overline \Sigma \setminus \Sigma\right|\varrho.
\]
Since $\varrho>0$ was arbitrary, the lemma follows.
\end{proof}

\section{Proofs of main theorems}
\label{_proof_Section_}

We continue using the notation in the beginning of Section~\ref{sec:hc}, and we write $[\overline\Sigma,(M,l)]^{as}$ for the set of almost smooth homotopy classes of maps as in Definition~\ref{df:as}.
\begin{proof}[Proof of Theorem~\ref{tm:main}]
Let $\Sigma$ be a nice Riemann surface and let $l$ be an $A$-labeling of $\Sigma.$ For $\Theta = xJ + yK \in R_I,$ write $\omega_\Theta = x\omega_J + y \omega_K,$ and let $I_{\omega_\Theta} : [\overline\Sigma,(M,l)]^{as}\to \R$ be as in Lemma~\ref{lm:as}. For
\[
h \in [\overline\Sigma,(M,l)]^{as},
\]
define $\phi_{h,\Sigma} : R_I \to \R$ by
\[
\phi_{h,\Sigma}(\Theta) = I_{\omega_\Theta}(h).
\]
Let $V$ be the vector space consisting of real linear combinations of $\omega_J$ and $\omega_K$. Define
\[
\tilde\phi_{h,\Sigma} : V \to \R
\]
by $\tilde\phi_{h,\Sigma}(p\omega_J + q \omega_K) = p I_{\omega_K}(h) + q I_{\omega_J}(h)$. Since $\tilde\phi_{h,\Sigma}$ is linear on $V$, the restriction of $\tilde \phi_{h,\Sigma}$ to the circle $\widetilde R_I = \{p \omega_J + q \omega_K|p^2 + q^2 = 1\}$ is either identically zero or has a unique maximum.  Since $\phi_{h,\Sigma}(\Theta) = \tilde\phi_{h,\Sigma}(\omega_\Theta)$, the same holds for $\phi_{h,\Sigma}.$ Let $P \subset R_I$ be the set of $\Theta$ such that
\[
\phi_{h,\Sigma}(\Theta) = \max_{R_I} \phi_{h,\Sigma} > 0
\]
for some nice $\Sigma,$ some $A$-labeling $l$ and some almost smooth homotopy class $h \in [\overline\Sigma,(M,l)]^{as}.$ Since the set of topological types of $\Sigma$ and choices of $l$ is countable, Lemma~\ref{lm:count} implies $P$ is countable.

Let $\Theta_0 \in R_I\setminus P,$ let $\Sigma$ be a nice Riemann surface, and let $l$ be an $A$-labeling. Let $u: \Sigma \to M$ be $\Theta_0$-holomorphic with $u(\partial\Sigma_i) \subset L_{l(i)}$ and $E_g(u) < \infty.$ We prove that $u$ is constant. Indeed, the condition $E_g(u) < \infty$ and Theorem~\ref{tm:conti} imply that $u$ extends to a continuous map
$\bar u \in \Map(\overline \Sigma,(M,l))$ that is almost smooth.
Let
\[
h_u:= [\overline u] \in [\overline\Sigma,(M,l)]^{as}
\]
be the homotopy class represented by $\bar u.$
By Lemma~\ref{lm:ei},
\[
\phi_{h_u,\Sigma}(\Theta) = \int_\Sigma u^*\omega_\Theta \leq E_g(u)
\]
with equality if and only if $u$ is $\Theta$-holomorphic, that is $\Theta = \Theta_0.$ So, $\phi_{h_u,\Sigma}$ achieves its maximum at $\Theta_0$. Since $\Theta_0 \notin P,$ it follows that
\[
E_g(u) = \phi_{h_u,\Sigma}(\Theta_0) = 0.
\]
So, $u$ must be constant.
\end{proof}

\begin{proof}[Proof of Theorem~\ref{tm:loc}]
Let $\Theta \in (R_I \setminus P)\cap \J_\omega$. If Theorem~\ref{tm:loc} were false, we could find a sequence of nice Riemann surfaces $\Sigma^j$ with $A_0$-labelings $l_j$ and
\begin{gather*}
u_j : \Sigma^j \to M, \qquad w_j \in \Sigma^j, \qquad t_j \in [0,1/j]
\end{gather*}
such that
\begin{gather*}
\bar\partial_\Theta u_j = 0, \qquad u_j((\partial\Sigma^j)_i) \subset \phi_{l_j(i),t_j}(L_{l_j(i)}),\\ E_g(u_j) \leq E_0,\qquad
\chi_c(\Sigma^j) \geq \chi_0, \qquad u_j(w_j) \in M \setminus V.
\end{gather*}
By Gromov compactness, we would have the following:
\begin{enumerate}
\item
A real number $E_\infty \in [0,E_0]$ and an integer $\chi_\infty \geq \chi_0.$
\item
A labeled graph $\Gamma$ consisting of the following data:
\begin{itemize}
\item
A finite set of vertices $V.$
\item
A finite set of half-edges $H.$
\item
A fixed point free involution $\sigma: H \to H.$ The set of orbits of $\sigma,$ denoted by $E,$ is called the set of edges of $\Gamma.$
\item
A map $\nu : H \to V$ sending each half-edge to the vertex to which it is attached.
\item
Distinguished vertices $v_0,v_\infty \in V.$
\item
A map $E : V \to \R_{\geq 0}$ such that
$
\sum_{v \in V} E(v) = E_\infty.
$
\item
A map $\chi : V \to \Z$ such that
$
\sum_{v \in V} \chi(v) = \chi_\infty.
$
\end{itemize}
\item
For each $v \in V \setminus \{v_0\},$ a connected closed Riemann surface $\Sigma^v$ and a $\Theta$-holomorphic map
\[
u_v : \Sigma^v \to M
\]
with $E_g(u_v) = E(v)$ and $\chi(\Sigma^v) - |\nu^{-1}(v)| = \chi(v).$
\item
A nice Riemann surface $\Sigma^{v_0}$ with $\chi_c(\Sigma^{v_0}) - |\nu^{-1}(v_0)| = \chi(v_0)$ and an $A_0$-labeling $l_\infty.$
\item
A $\Theta$-holomorphic map
$
u_{v_0} : \Sigma^{v_0} \to M
$
with
\[
u_{v_0}((\partial\Sigma^{v_0})_i) \subset L_{l_\infty(i)}
\]
and $E_g(u_{v_0}) = E(v_0).$
\item
For each $h \in H,$ a point $w_h \in \Sigma^{\nu(h)}$ such that for each edge $e = \{h,h'\} \in E,$ we have $u_{\nu(h)}(w_h) = u_{\nu(h')}(w_{h'}).$
\item
A point $w_\infty \in \Sigma^{v_\infty}$ such that $u_{v_\infty}(w_\infty) \in M\setminus V.$
\end{enumerate}
On the other hand, Theorem~\ref{tm:main} asserts that $u_v$ is constant for all $v \in V,$ so
\[
u_{v_\infty}(w_\infty) \in \bigcap_{i \in [b(\Sigma^{v_0})]} L_{l_\infty(i)} \subset V,
\]
which is a contradiction.
\end{proof}

\begin{rem}\label{rem:fgc}
The proof of Theorem~\ref{tm:loc} uses a partial version of Gromov compactness that keeps track of only one of the components with boundary of the limiting stable map. Full Gromov compactness that keeps track of all components with boundary of the limiting stable map would be preferable. Indeed, if the topological type of $\Sigma$ and the labeling $l$ were fixed, full compactness would allow us to replace the hypothesis
\[
\bigcup_{i \in A_0} L_i \subset V
\]
with the weaker hypothesis
\[
\bigcap_{i \in [b(\Sigma)]} L_{l(i)} \subset V.
\]
Full compactness should indeed be true if we assume the Hamiltonian flows $\phi_{i,t}$ are real analytic. The proof would use an argument similar to the proof of Theorem~\ref{tm:conti} to control the behavior of low energy holomorphic strips and thus show that components with boundary connect. We leave this for future work.
\end{rem}

\section*{Acknowledgements}
We are grateful to M. Abouzaid, S. Galkin, P. Seidel, I. Smith and G. Tian, for helpful discussions. J.S. was partially supported by ERC starting grant 337560 and ISF Grant 1747/13. M.V. was partially supported by the Russian Academic Excellence
Project `5-100' and CNPq - Process 313608/2017-2.

\hfill

\noindent {\sc Jake P. Solomon\\
Institute of Mathematics \\
Hebrew University \\
Givat Ram, Jerusalem, Israel,}\\
{\tt jake@math.huji.ac.il}

\hfill

\noindent {\sc Misha Verbitsky\\
{\sc Instituto Nacional de Matem\'atica Pura e
              Aplicada (IMPA) \\ Estrada Dona Castorina, 110\\
Jardim Bot\^anico, CEP 22460-320\\
Rio de Janeiro, RJ - Brasil }\\
also:\\
{\sc Laboratory of Algebraic Geometry,\\
National Research University HSE,\\
Department of Mathematics, 7 Vavilova Str. Moscow, Russia,}\\
\tt  verbit@mccme.ru}.

\end{document}